\documentclass[a4paper]{amsart}
\usepackage{amsaddr}

\usepackage{graphicx}
\usepackage{amsmath}
\usepackage{amsthm}
\usepackage{hyperref}
\usepackage[all,rotate,2cell]{xy}
\usepackage{amsfonts}
\usepackage{amssymb}
\usepackage{tikz}
\usepackage{tikz-cd}
\usepackage{url}
\usepackage{verbatim}

\UseTwocells

\newtheorem{theorem}{Theorem}

\newtheorem{proposition}[theorem]{Proposition}

\theoremstyle{definition}
\newtheorem{definition}[theorem]{Definition}
\newtheorem{example}[theorem]{Example}
\newtheorem{remark}[theorem]{Remark}

\newcommand{\eg}{eg}
\newcommand{\ie}{ie}
\renewcommand{\epsilon}{\varepsilon}
\newcommand{\Forall}[1]{\forall #1 \ldotp}
\newcommand{\Exists}[1]{\exists #1 \ldotp}

\newcommand{\catname}[1]{\mathcal{#1}}

\newcommand{\catC}{\catname{C}}

\newcommand{\catE}{\catname{E}}
\newcommand{\catF}{\catname{F}}

\newcommand{\baseS}{\catname{S}}             
\newcommand{\Set}{\mathsf{Set}}   

\newcommand{\Sierp}{\mathbb{S}}              
\newcommand{\twopos}{\overrightarrow{2}}    
\newcommand{\Nat}{\mathbb{N}}       
\newcommand{\Qrat}{\mathbb{Q}}      
\newcommand{\Reals}{\mathbb{R}}              
\newcommand{\Spec}{\mathop{\mathsf{Spec}}}   
\newcommand{\Clop}{\mathsf{Clop}}   
\newcommand{\pt}[1]{\mathsf{pt}(#1)}                
\newcommand{\GRD}{\mathsf{GRD}}     
\newcommand{\site}{\mathsf{Site}}   
\newcommand{\power}{\mathcal{P}}            
\newcommand{\fin}{\mathcal{F}}              
\newcommand{\List}{\mathop{\mathsf{List}}}  
\newcommand{\cons}{\mathsf{cons}}
\newcommand{\Idl}{\mathsf{Idl}}  

\newcommand{\unmeet}{\mathop{\bigwedge}}   
\newcommand{\unjoin}{\mathop{\bigvee}}   
\newcommand{\covered}{\vartriangleleft}   
\newcommand{\turnstile}[1]{\mathrel{\mathop{\vdash}_{#1}}}

\newcommand{\Id}{\mathop{\mathsf{Id}}}

\newcommand{\DL}{\mathsf{DL}}               
\newcommand{\Loc}{\mathbf{Loc}}             
\newcommand{\GFr}{\mathfrak{GFr}}           
\newcommand{\BTop}{\mathfrak{BTop}}         

\newcommand{\thT}{\mathbb{T}}               
\newcommand{\thobj}{\mathbb{O}}             
\newcommand{\oftype}{\mathord{:}}
\newcommand{\inj}{\mathsf{inj}}

\newcommand{\obj}{\mathsf{ob}}   
\newcommand{\mor}{\mathsf{mor}}  
\newcommand{\homsp}[2]{^{#1}_{#2}} 

\newsavebox{\vctikzbox}   
{\begin{lrbox}{\vctikzbox}\begin{tikzpicture}[#1]}%
	{\end{tikzpicture}\end{lrbox}
	\raisebox{-.5\height}{\usebox{\vctikzbox}}}

{\begin{lrbox}{\vctikzbox}\begin{tikzcd}[#1]}%
	{\end{tikzcd}\end{lrbox}
	\raisebox{-.5\height}{\usebox{\vctikzbox}}}

\begin{document}

\title{Generalized point-free spaces, pointwise}
\author{Steven Vickers}
\address{School of Computer Science, University of Birmingham}
\email{s.j.vickers@cs.bham.ac.uk}


\maketitle
\begin{abstract}
	We survey foundational principles of Grothendieck's generalized spaces, including a critical glossary of the various, and often conflicting, terminological usages.
	
	Known results using generalized points support a fully pointwise notation for these essentially point-free structures. This includes some from dependent type theory to deal with bundles as continuous space-valued maps, mapping base point to fibre.

\end{abstract}

\newcommand{\glit}[1]{                
	
	\vspace{.5\baselineskip}
	\hspace{-2\parindent}
	\textbf{#1}:}

\newcommand{\glemph}[1]{\textbf{#1}} 

\section{Introduction}
\label{sec:Intro}
Point-free topology is well established by now, as an alternative approach to topology that is valid in non-classical models of sets such as elementary toposes, and there gives better mathematics than the traditional point-set topology.

However, it is often misunderstood as being a ``pointless'' treatment that avoids discussing the points and instead works entirely with algebraic structures such as frames in locale theory \cite{StoneSp}, or presentations of frames in formal topology \cite{Sam87}.
A grand generalization of this is Grothendieck's uncovering of generalized topological spaces for which the algebraic structures are toposes.

Actually, point-free topology is not pointless. It has been known since the formulation of classifying toposes that it is frequently possible, and more intuitive, to reason validly in a pointwise manner.

The aim of these notes is to bring together in a clear form the principles used (which are mostly not new), and develop some more systematic notations for manipulating the spaces and their points. In particular, we introduce a fibrewise notation for bundles.

At the same time, we do our best to hijack the language of conventional topology so that a point-free exposition reads as if it is proper topology, and not a distinct discipline.
In this we follow the lead of Joyal and Tierney~\cite{JoyalTier}, who boldly used the word ``space'' for locale. 
Of course, our contention is that it is the point-\emph{set} approach that is not quite proper topology.

The core mechanism that enables pointwise reasoning here is that all constructions must be done \emph{geometrically}, using colimits and finite limits.
Much has been accomplished in this style, but we still do not know the full scope of such geometric mathematics.
Hence this article is not so much finished work as an invitation to further research. Section~\ref{sec:Conc} discusses aspects of this.

In Section~\ref{sec:Glossary} we have provided a glossary that discusses our terminology, and our reasons for choosing it. Any word or phrase \glemph{emphasized thus} has an entry there.

\section{Point-free topology}
\label{sec:PtFree}

The various understandings of the phrase ``point-free topology'' (some canonical examples such as locales and formal topologies are discussed in Section~\ref{sec:Examples}) can be subsumed in the following definition. 

\begin{definition}
\label{def:ptfree}
\label{def:map}
	A \emph{(point-free topological) space} is defined by a \glemph{geometric theory}; its \emph{points} are the models of the theory.	
	A map $f\colon X \to Y$ is then defined by a geometric construction of points $f(x)\oftype Y$ out of points $x\oftype X$.
	
	Suppose $f$ and $g$ are two maps from $X$ to $Y$.
	Then a \emph{specialization morphism} from $\alpha\colon f\to g$ is defined by a geometric construction of homomorphisms $\alpha(x)\colon f(x)\to g(x)$ out of points $x\oftype X$.
	(A homomorphism between models of a geometric theory is given by a carrier function for each sort, all preserving the function and predicate symbols.)
	
	The formal parameter $x$ is a \glemph{generic} point,
	concretely represented in the category $\baseS X$ of \glemph{sheaves} over $X$.
	Any \glemph{specific} point can be used to instantiate (substitute for) it, so thus behaving as an actual parameter.
\end{definition}

By contrast, a \emph{point-set space} is defined by a set of points and a topology in the usual way.
Since our aim is to adopt the language of ordinary topology, the point-free interpretation is our default for spaces and topology. If we wish to refer to point-set spaces, we shall be explicit about it.

If $\thT$ is a geometric theory, we write $[\thT]$ for the corresponding space. In reality this is just a formal notation, but the \emph{maps}, the space morphisms, are different from theory morphisms.

Note that a geometric theory simultaneously specifies both the points (the models) and the topology, the opens being the propositions modulo equivalence.
Unlike the case with a point-set space, the points and topology are not -- indeed, cannot be -- separated out.%
\footnote{
	It is true that, in elementary toposes, each locale has a \glemph{discrete coreflection} that serves as its ``set of points''.
	However, that construction is not geometric, and so can play no useful part in our development.
}

The geometric theory here may be a first-order predicate theory, which means the spaces may be \glemph{generalized},
\ie\ \glemph{toposes}.
The \emph{ungeneralized} point-free spaces are the \glemph{localic} ones.

The models are not restricted to those in \glemph{$\baseS$}, one's preferred \glemph{base} category of sets.
They can be sought in any \glemph{Giraud frame} (= Grothendieck topos = logos), and, correspondingly, the \glemph{points} are ``generalized'' points.

Remarkably, the infinite disjunctions in a \glemph{geometric theory} allow us to \glemph{syntacticize} \glemph{geometric constructions} (Section~\ref{sec:Syntactn}),
and it follows that we make no change to their expressive power if we allow sort constructors for geometric constructions.
This style of extending geometric theories (see Section~\ref{sec:GeoCons}) is discussed in~\cite{LocTopSp} and is closer to the definition of geometric theory in~\cite[B4.2.7]{Elephant1}.
It liberates us from any bureaucratic need to encode the geometric constructions in a pure logic, and we shall use it extensively.
(Contrast it with the duality theorem of ~\cite[3.2.5]{Caramello:TST}, which takes great care to stay within a single signature in the strict logical syntax.)

A consequence of this liberal notion is that first-order theories are still localic, as long as their sorts are built geometrically ``out of nothing''.
An illuminating example in~\cite{LocTopSp} is the real line $\Reals$.
The geometric theory can be given in propositional form, which corresponds to the definition in~\cite{StoneSp}.
However, more natural is a first-order theory that directly models Dedekind sections of the rationals $\Qrat$.
Since $\Qrat$ can be constructed geometrically, it can be used as a sort without spoiling the localic nature of the theory.

\section{Glossary}
\label{sec:Glossary}

A significant part of this paper is to propose a language of point-free topology that can be read as topology, avoiding signals, such as ``sublocale'' for ``subspace'', or ``geometric morphism'' for ``map'', that mark it as something different from what topologists study.
In this section we discuss the language proposed, with reasons for our choices.

At the same time, it is a convenient place to summarize the features of topos theory that we use.
The main references are \cite{MacMoer,Elephant1,Elephant2}; see also \cite{LocTopSp} for a readers' guide through some of their pointwise aspects, and \cite{AnelJoyal:Topologie} for a more recent account that attempts to be careful about the terminology
-- as well as accessing higher dimensions.
For the localic spaces, see \cite{StoneSp,JoyalTier}, and \cite{TVL} for a more elementary introduction.

\glit{arithmetic theory}
See \glemph{theory}.

\glit{arithmetic universe (AU)}
(Originally Joyal; developed \cite{Maietti:AritJ,ArithInd,Vickers:AUSk,Vickers:AUClTop})
Pretopos with parametrized list objects.
Their internal mathematics is a small fragment of the geometric, with finite limits, \emph{finite} colimits, and free algebra constructions.
They seem to occupy a logical sweet spot in that their structure is self-internalizable (hence Joyal's application to G\"odel's Theorem). Their theory is essentially algebraic, hence amenable to standard results in universal algebra, but at the same time their internal maths supports that universal algebra.

Using them, a good fragment of geometric reasoning can be done independently of base $\baseS$ \cite{Vickers:AUClTop},
and this calls into question the reliance on \glemph{formal duals} as definition of space.  

\glit{base (category of sets)}
See \glemph{$\baseS$}.

\glit{$\BTop/\baseS$}
2-category of bounded $\baseS$-toposes, = \glemph{$\GFr$} relative to $\baseS$.

\glit{bundle}
A map viewed as a construction of fibres from base points.
In point-set topology it is non-trivial to get the map from the fibre construction, but in geometric point-free topology it is automatic because of continuity properties of the geometric construction.
See Section~\ref{sec:Bundles}.

\glit{classifying category}
-- for a theory $\thT$. This is a category (with structure) equipped with a \glemph{generic} model to make it a representing object for categories equipped with models of $\thT$.
The structure needed in the category will depend on the structure assumed when presenting $\thT$:
for example, finite products for algebraic theories (in the single-sorted case, the classifying categories are better known as \emph{Lawvere theories}), and finite limits for finite limit theories.

It can also be thought of as the category (with structure) freely presented by $\thT$, with the signature as generators and the axioms as relations.
This can be done by universal algebra, provided the structure involved is finitary~\cite{PHLCC}.
A particular appeal of \glemph{arithmetic universes} is that their classifying categories exist by these finitary means, even though their natural number objects enable them to capture a surprisingly expressive fragment of geometric logic in a finitary way. 

See also~\cite{Crole:CatTypes}.

\glit{classifying topos}
-- for a geometric theory $\thT$.
These are the \glemph{classifying categories} in the geometric case.
Because of the infinitary operations (disjunctions, coproducts), they cannot be constructed by universal algebra.
However, they do exist, and can be constructed in two steps using a \glemph{syntactic category} and a category of \glemph{sheaves} over a coverage. (The second step introduces the colimits that are absent in a syntactic category.)

The definition of ``classifying topos'' is very much in terms of the \glemph{Giraud frame}, and assumes a fixed base $\baseS$.
Then, by \glemph{formal duality}, its \glemph{points} are just the models of $\thT$ (Section~\ref{sec:FormDual}),
and we can understand the classifying topos as the formal dual of ``the \glemph{space of models} of $\thT$''.

In our notation, which is close to standard, the classifying topos of $\thT$ is $\baseS [\thT]$.
This can be read in two ways: either as the Giraud frame generated from $\baseS$ by adjoining the ingredients of $\thT$,
or as the category of $\baseS$heaves over the space $[\thT]$ of models of $\thT$.

\glit{coherent theory}
See \glemph{theory}.

\glit{discrete space}
The space of elements of a \glemph{set}.

\glit{discrete coreflection}
The \glemph{discrete} coreflection of a locale is its ``set of points''.
It can be constructed in any elementary topos, but not geometrically.

Since~\cite{JoyalTier} an internal set (object) corresponds to a local homeomorphism, and an internal frame to a general localic bundle,
the discrete coreflection is approximating a bundle by a local homeomorphism.
Example~\ref{ex:SierpClosed} shows how badly the approximation can lose information. There a non-trivial space has empty set of points, not for reasons of logic, but for purely topological reasons.

\glit{elephant theory}
See \glemph{theory}.

\glit{$\fin$}
The Kuratowski finite powerset, equivalently the free semilattice (under $\cup$).

\glit{formal dual}
Defining things as objects of the dual category of some other things, for examples \glemph{locales} as formal duals of \glemph{frames}. Usually a different language is introduced for the dual object, often (not always) a different mathematical notation.

A major example of this style is the \glemph{topos}-\glemph{logos} duality as described by Anel and Joyal in~\cite{AnelJoyal:Topologie}, where they reclaim the language of spaces for toposes, and introduce a different language of ``logoi'' for the categories (of sheaves) that represent them.
For them a topos $\mathcal{X}$ ``is the same as'', but is notationally differentiated from, its formal dual logos $\mathcal{S}\mathrm{h}(\mathcal{X})$.
(In our notation they correspond to the space $X$ and the \glemph{Giraud frame} $\baseS X$.)

Their account of topos as formally dual to (what they now call) logos is historically accurate, and provides the starting point for a discussion of higher-dimension $\infty$-toposes.
Nonetheless, there are a couple of problems with the use of formal duality as a core part of the narrative,
and the aim of the present article is to propose a way round them, at least for ordinary toposes.

The first problem is that it makes a technical distinction between a locale $X$, which ``is'' a frame, and the corresponding localic topos, which ``is'' a logos.
That distinction is mathematically inessential -- which in fact Anel and Joyal recognize (Section 3.1.1) by using the same notation for both, so that $X$ ``has'' both a frame $\mathcal{O}(X)$ and a logos $\mathcal{S}\mathrm{h}(X)$, just as a point-set space does.

The second problem, and more serious, is that for a classifying topos $[\thT]$, the logos (which we write $\baseS[\thT]$) depends on $\baseS$: as either foundational assumptions about one's ambient logic, or an explicit choice of a base $\baseS$.
For \glemph{arithmetic theories} $\thT$, for example the object classifier, for which $\thT$ has one sort and no other symbols or axioms, there is much we can say about the point-free space $[\thT]$ without knowing or caring what $\baseS$ is.
It has an existence independent of base; but the account by formal duality fails to capture that.
(cf.~\cite{Vickers:AUClTop}. There the independent existence leads a simple life in the 2-category of AU-contexts, while the topos account using formal duality is a highly elaborate structure fibred over a 2-category of possible bases.)

\glit{frame}
Complete lattice with binary meet distributing over all joins.
The notion can be internalized in any \glemph{elementary topos}, with joins in a frame $A$ being described by a morphism $\power A=\Omega^A \to A$ 
-- see~\cite{JoyalTier,Elephant2}.%
\footnote{
	The Elephant references are somewhat scattered. Sections~C1.6 and~C2.4 are the most relevant; see also~B2.3 for internal cocopleteness.
}

\glit{generalized (space)}
The corresponding geometric theory may be first order: thus generalized in the sense of topos as generalized topological space.%
\footnote{
	See Garner's \emph{ionads}~\cite{Garner:Ionads} for an account of what the corresponding point-set generalization might be.
}
This is the default.
cf. \glemph{localic}.

\glit{generalized (point)}
See \glemph{point}.

\glit{generic}
(point or model) The model with which a \glemph{classifying category} is equipped to make a representing object for the category of models.
Intuitively, it is a model ``with no properties'', other than the logical consequences of its being a model, and any construction on it can be instantiated for any \glemph{specific} point.

\glit{geometric construction}
One preserved by inverse image functors,
built up using colimits and finite limits.
An important class of them is the free algebra constructions (see also \glemph{arithmetic universe}).

\glit{geometric theory}
The most familiar form of geometric theory is described in~\cite[D1.1.6]{Elephant2}. It is a multi-sorted first-order theory whose axioms take the form of sequents in context,
$\phi\turnstile{x,y,z,\ldots}\psi$, where the geometric formulae $\phi$ and $\psi$ are built using finite conjunctions $\top,\wedge$, disjunctions $\unjoin$ (possibly infinite -- the \glemph{base} $\baseS$ determines what ``infinities'' are), existential quantification $\exists$, and equality $=$.
\cite{LocTopSp} discusses informally the possibility of including geometric sort constructors as a way to increase the flexibility of expression in a geometric theory, eg by having sorts required to be isomorphic to the integers or the rationals. This is formalized in~\cite[B4.2.7]{Elephant1} and (for the \glemph{arithmetic} fragment) in \cite{Vickers:AUSk}, using sketches. Both build up a theory in a stepwise manner.
We discuss this in more detail in Section~\ref{sec:GeoCons}.

\glit{$\GFr$}
2-category of \glemph{Giraud frames}.

\glit{Giraud frame}
(\cite[B4.5]{Elephant1}, following~\cite{TopCat})
Category satisfying the conditions set out in \glemph{Giraud's Theorem}.
(See also \glemph{topos -- Grothendieck} and \glemph{logos}.)
It has finite limits and ``small'' colimits, suitably behaved, and, though large, it is small enough that it can be generated from a small structure.

Allowing for non-classical \glemph{base} $\baseS$, and using $\baseS$ to govern what ``small'' means \cite[B3.3]{Elephant1}, it comes out as a bounded $\baseS$-topos.

A morphism of Giraud frames is a functor preserving finite limits and small colimits, and comes out as the inverse image half of a geometric morphism.

The terms ``Giraud frame'', ``Grothendieck topos'' (relative to $\baseS$) and ``logos'' have roughly the same meaning. In this paper we shall give preference to ``Giraud frame''; nonetheless, ``logos'' has its merits. 

\glit{Giraud's Theorem}
A category is a \glemph{Giraud frame} iff it is equivalent to the category of sheaves over an internal site in the \glemph{base} $\baseS$.

The original version took $\baseS$ to be the category of classical sets. It has since been generalized to arbitrary $\baseS$ (see~\cite[B3.3.4]{Elephant1}, \cite[C2.4.6]{Elephant2}).

\glit{global (point)}
See \glemph{point}.

\glit{Lindenbaum algebra}
For a propositional theory, this is the algebra of formulae modulo equivalence provable from the axioms of the theory.
The algebraic operations correspond to the logical connectives.
For a first order theory we take the \glemph{classifying category} as the appropriate generalization. 

\glit{locale, localic space}
A space equivalent to one corresponding to a propositional geometric theory, one for which the signature has no sorts.
Then there can be no terms or variables (since they have to have sorts), so the rest of the signature can have no function symbols, and its predicate symbols can have no arguments: they are just propositions.

To avoid misunderstanding, we follow neither~\cite{StoneSp} (locale = \glemph{formal dual} of \glemph{frame}) nor~\cite{JoyalTier} (locale = frame).

\glit{logos}
\cite[Section 3.3]{AnelJoyal:Topologie}
Roughly speaking, these are \glemph{Giraud frame}s -- though, foundationally, there may be room for uncertainty over how precise that is.
See \glemph{formal dual}.

\glit{model}
Understood in the sense of categorical logic, so that models can be sought in any category with sufficient structure to support the logic, not just $\baseS$ (cf. \glemph{point}). The signature (sorts, functions, predicates) is interpreted as objects, morphisms and subobjects; and then the axioms have to be respected.
See~\cite[D1.2.12]{Elephant2}.

\glit{open}
Sub\glemph{sheaf} of 1.

\glit{point}
  A point of $X$ is understood in a generalized sense, as a morphism with codomain $X$. A \glemph{global} point is one whose domain is 1.
  
  The category of point-free spaces is not well-pointed. (This comes down the fact that geometric logic is incomplete. There are not, in general, enough models to support all the inequivalences between geometric formulae.) For that reason, the pointwise reasoning that we are describing here cannot be in terms of global points, so by default all our points are in the generalized sense.
  
  cf. the distinction between generalized and global elements in an elementary topos, as used in the Kripke-Joyal semantics (see~\cite{MacMoer}).
  
  See also \glemph{generic} and \glemph{specific}.

\glit{pointwise/pointless}
With/without reference to points.

\glit{point-free/point-set}
Points described as models of a geometric theory (the default)/elements of a set.

\glit{$\pt{-}$}
We frequently use this as notation for a \glemph{space} constructed in some understood way out of some structure.
Let us emphasize that here it is not meant as the collection of \glemph{global} points. 

\glit{$\baseS$}
An \glemph{elementary topos} (with nno) understood as the base over which \glemph{Giraud frames} are constructed.
For a fixed choice of $\baseS$, spaces can be understood as the \glemph{formal dual}s of Giraud frames, with space $[\thT]$ corresponding to the \glemph{classifying topos} $\baseS[\thT]$.

\glit{set}
An object in a \glemph{Giraud frame}. Thus the object classifier $\Set$ is the ``space of sets''.
See \glemph{discrete space}.

\glit{sheaf}
A sheaf on $X$ is a map from $X$ to the space of \glemph{sets}. Equivalent (obviously) to an object of $\baseS X$; also \cite{JoyalTier} to a local homeomorphism with codomain $X$.

\glit{space of models}
($[\thT]$, for a geometric theory $\thT$)
If we are assuming a fixed \glemph{base} $\baseS$, then we can understand the space of models to be the \glemph{formal dual} $[\thT]$ of the \glemph{classifying topos} $\baseS[\thT]$.

If the \glemph{theory} $\thT$ is \glemph{arithmetic}, then we have the option of understanding $[\thT]$ in a base-independent way using a \glemph{classifying AU}~\cite{Vickers:AUClTop}.  

\glit{specific}
(point or model) Just a \glemph{point} or \glemph{model}, but by contrast with the \glemph{generic} point or model.
A specific point has more properties, specific to it.

\glit{Stone space}
The spectrum $\Spec A$, \ie\ the space of prime filters (ultrafilters), of a Boolean algebra $A$.
$A$ can be recovered as $\Clop (\Spec A)$, the lattice of clopens; and the frame $\Omega (\Spec A)$ is the ideal completion $\Idl A$.

\glit{syntactic category}
\cite[D1.4]{Elephant2}
This is a purely logical analogue of \glemph{classifying category}, for the case where the objects generated by a theory are all subobjects of products of carriers, and can thus be represented by formulae in context.
Thus we do not get, \eg, coproducts or list objects.

\glit{syntacticize}
We use this term to mean the reconciliation of ordinary mathematical notation for \glemph{geometric constructions} with the syntax of geometric logic.
In principle (Section~\ref{sec:Syntactn}), the infinitary disjunctions allow geometric constructs to be characterized up to isomorphism by first-order structure and axioms. However, this is often inconvenient, and a second mode of syntacticization (Section~\ref{sec:GeoCons}) is to adapt the logical syntax in a more type-theoretic way.

\glit{theory -- arithmetic}
A theory expressed using the fragment of geometric logic corresponding to \glemph{arithmetic universes}. It has been formalized in sketch form, using the \emph{AU contexts} of~\cite{Vickers:AUSk}.
See Section~\ref{sec:ArithTh}.

\glit{theory -- coherent}
Like a \glemph{geometric theory}, but in which all disjunctions must be finitary.
Note -- the definition of ``geometric theory'' in~\cite[X.3]{MacMoer} appears to encompass only coherent theories.

\glit{theory -- elephant}
A theory (with respect to $\baseS$) described as in~\cite[B4.2.5]{Elephant1}, by specifying all models in all \glemph{Giraud frames}.
See Section~\ref{sec:ElephantTh}.

\glit{theory -- geometric}
See \glemph{geometric theory}.

\glit{topos} This is a difficult word. Grothendieck said a topos is a generalized topological space, and proposed the name as conveying ``that of which topology is the study''. Our aim here is to capture a notion of space close to Grothendieck's vision.
(Note, however, that there is a huge open question of whether the geometric techniques can capture Grothendieck's applications in algebraic geometry,
which so often use non-geometric mathematics to construct the sites.)

In a sense, the word ``topos'' is self-defeating. As soon as we say, ``Let $X$ be a topos,'' a topologist will reject what follows as being topos-theory, not topology.
In any case, various problems have accumulated that have  polluted the word.
First, there is an unresolvable tension between saying the topos is a generalized space, while at the same time it ``is'' the category of sheaves.
To say ``the topos of sets'' will almost certainly not be understood as ``the generalized space of sets'', \ie\ the object classifier. 
Moreover, the now widespread use of ``topos'' to mean elementary topos has lost the spatial meaning. In itself, an elementary topos is not a generalized space
-- the spaces are geometric morphisms.
All this motivated Anel and Joyal \cite{AnelJoyal:Topologie} to introduce ``\glemph{logos}'',
to allow ``topos'' to retreat to its proper meaning.

Despite the fact that what we are doing is fundamentally topos theory, we have not found a good place for the word ``topos'' in our vocabulary.
Its use here will mainly be to clarify the topos-theoretic basis of what we do.

\glit{topos -- elementary}
The usual definition, except that we assume all our elementary toposes have a natural numbers object (nno).
Without that, we cannot classify first-order geometric theories~\cite[B4.2.11]{Elephant1}.

\glit{topos -- Grothendieck}
We shall use this as synonymous with \glemph{Giraud frame}.
Frequently it is also used as in Grothendieck's original treatment to mean a Giraud frame relative to the category of classical sets (or, more carefully, a Grothendieck universe).
Thus, although the notion of Grothendieck topos makes sense relative to any base $\baseS$, the phrase is too easily read more narrowly as being specific to an ambient classical set theory. We shall generally avoid it.

\glit{ungeneralized (space)}
= \glemph{localic}

\section{Canonical examples}
\label{sec:Examples}
Here we describe more familiar descriptions of point-free topology, and explain how they match Definition~\ref{def:ptfree}.

In each case, given a presenting structure $\Sigma$, we define a space $\pt{\Sigma}$ of points.
We stress again that this is the space corresponding to a geometric theory, \emph{not} the set of global points.

\subsection{Frames}
\label{sec:Frames}

The global points corresponding to a \glemph{frame} $A$ are the frame homomorphisms from $A$ to $\Omega$. \cite{StoneSp} gives several alternative descriptions,
of which the most important geometrically is that of \emph{completely prime filters:} the function $A\mapsto\Omega$ gives a subset $F$ of $A$ -- the ``true kernel'' --, and then the preservation of finite meets and arbitrary joins translates into (respectively) the filter property and complete primeness.
We can now present a geometric theory of completely prime filters of $A$.

\begin{definition}
	\label{def:compPrimeFfilters}
	Let $A$ be a frame. Then the space $\pt {A}$ corresponds to the following propositional geometric theory.
	
	The signature (set of propositional symbols) is $A$ itself.
	To avoid confusion between the geometric logic and the (closely related) lattice structure, we shall here write $\phi_a$ for the propositional symbol corresponding to $a\in A$. (Elsewhere we may be less fussy.)
	
	The axioms fall into three families:
	\[  \begin{array}{ll}
		\phi_a \vdash \phi_b & (a\leq b) \\
		\bigwedge_{a\in S}\phi_a \vdash \phi_{\unmeet S}
		& (S\in \fin A) \\
		\phi_{\unjoin U}\vdash \unjoin_{a\in U}\phi_a
		& (U\in\power A)
	\end{array}  \]
	($\fin A$ denotes the Kuratowski finite powerset of $A$, the free semilattice on the set $A$.
	Here it would in fact suffice to restrict $S$ to be empty or a doubleton: the general $n$-ary case follows from the nullary and binary.)
\end{definition}

A model of this theory can be analysed as a completely prime filter as follows.
First, each proposition $\phi_a$ in the signature has to be interpreted, as a truthvalue. This gives gives our $F\subseteq A$.
Then the first family of axioms says that if $a\leq b$ and $\phi_a$ is true, then $\phi_b$ is true. In other words, if $a\in F$ and $a\leq b$ then $b\in F$: $F$ is up-closed.
The second family says that if $\phi_a$ is true for every $a\in S$, then $\phi_{\unmeet S}$ is true.
In other words, if $S\subseteq F$, then $\unmeet S\in F$: $F$ is a filter.  
The third family says that if $\unjoin U\in F$ then $a\in F$ for some $a\in U$: $F$ is completely prime.

The key property of frames is that $A$ can be recovered from the geometric theory, as the \glemph{Lindenbaum algebra}.
In fact any geometric theory $\thT$ has a frame $\Omega[\thT]$ of opens, the subobjects of 1 in $\baseS[\thT]$.
In the Lindenbaum algebra, the frame structure is got from the disjunctions and finite conjunctions in the logic; but the axioms of the geometric theory force those to match the joins and finite meets in the original frame.

Taking this in conjunction with the fact that propositional geometric theories can be understood directly as frame presentations (see Section~\ref{sec:FrPres}),
we find that the frames provide canonical representatives of the spaces, and
we get an ability to treat localic point-free spaces as \glemph{formal duals} of frames.

Thus the frames appear to provide a highly satisfactory representation of the ungeneralized point-free spaces.
However, there is a big problem from the geometric point of view in that frame structure of $A$ depends on the map $\power A\to A$, and the power set $\power A$ is a non-geometric construction.

If $A$ is a frame in an elementary topos $\catF$, and $f\colon\catE\to\catF$ is a geometric morphism, then $f^{\ast}(A)$ is not necessarily a frame.
Obvious examples come from the case where $\catF$ is classical sets, and $A$ is the subobject classifier $\Omega$ there, which is just the 2-element set $2=1+1$.
$f^{\ast}(2)$ will still be $1+1$ in $\catE$, but this is rarely isomorphic to $\Omega$.

\cite{JoyalTier} show how to complete $f^{\ast}(A)$ to a frame $f^{\sharp}(A)$,
and this can be constructed in three steps. The first is to construct a geometric theory as in our Definition~\ref{def:compPrimeFfilters}. The second is to apply $f^{\ast}$ to all its ingredients, to get a geometric theory in $\catE$: so the signature is $f^{\ast}(A)$, and the axioms are indexed by $f^{\ast}(\leq_A+\fin A+\power A)$ -- and note that $f^{\ast}(\power A)$ is not the same as $\power(f^{\ast}(A))$. This in fact tells us to find points of $\pt{A}$ in $\catE$. The third step is to construct its \glemph{Lindenbaum algebra} in $\catE$.

\subsection{Formal topologies}
\label{sec:FormTop}

The Martin-L\"of school of predicative maths wished to reject powersets as being impredicative, and along with them the frames -- in fact the frame for a discrete space $X$ is just the powerset $\power X$.
These constraints led to a style, known as \emph{formal topology,} in which the space is given by a \emph{base} $B$ (of opens: every open is to be a join of basic opens) and a \emph{coverage.} The coverage comprises \emph{covers,} of the form $a\covered U$ ($a\in B$, $U\subseteq B$) to specify that the basic open $a$ is covered by those in $U$.
As we shall see, this too describes a geometric theory.

There are two main flavours. In the original form~\cite{Sam87}, the coverage describes \emph{all} valid covers, while in \emph{inductively generated} formal topologies~\cite{CSSV:IndGenFT}, the coverage specifies enough covers for the rest to be deduced. As in~\cite{CSSV:IndGenFT}, we shall assume that $B$ comes equipped with a partial order. This makes certain properties, such as meet stability, easier to stipulate.

\begin{definition}
	\label{def:FormTopPts}
	\newcommand{\B}{\mathcal{B}}
	Let $\B=(B, \leq, \covered)$ be an inductively generated formal topology. Then its space $\pt{\B}$ corresponds to the following propositional geometric theory.
	
	The signature is $B$.
	(Again, for clarity we write $\phi_a$ for the propositional symbol corresponding to $a$.)
	
	The axioms are
	\[  \begin{array}{ll}
		\phi_a \vdash \phi_b & (a\leq b) \\
		\top \vdash \unjoin_{a\in B}\phi_a \\
		\phi_a \wedge \phi_b \vdash \unjoin\{\phi_c\mid c\leq a, c\leq b\}
		& (a,b\in A) \\
		\phi_a\vdash \unjoin_{u\in U}\phi_u
		& (a\covered U)
	\end{array}  \]
\end{definition}
The middle two axioms here are for ``flatness''. If the poset $B$ had finite meets, then they would be equivalent to saying that that meet was preserved in the logic, somewhat as for a frame.
Another way to see it is that a model will be given by a subset $F\subseteq B$ that is a filter in a more general sense for posets: if $a,b\in F$, then there is some $c\in F$ with $c\leq a,b$. This filter then must ``split covers'' in the sense that if $a\covered U$ and $a\in F$, then $u\in F$ for some $u\in U$.
These are defined in~\cite{CSSV:IndGenFT} as the ``formal points''.%
\footnote{
	As always, our points, our models, may be defined in categories other than $\baseS$, our base category of sets.
	The ``formal'' points of \cite{CSSV:IndGenFT} seem to be restricted to lie in $\baseS$, in other words to be global points, but it cannot be that ``formal'' here means ``global''. Rather, I understand that ``global'' is implicit, and ``formal'' indicates that the points are in the sense of formal topology rather than point-set topology.
}

\subsection{Frame presentations, GRD-systems}
\label{sec:FrPres}

The formal topologies are an illustration of a broader principle, that we can work with presentations of frames without necessarily calculating the frame itself.

A frame presentation will have a set $G$ of generators, together with a set $R$ of relations, each of the form
\[
\text{frame word in } G \leq \text{frame word in }G
\text{.}
\]
Frame words are built from generators using joins and finite meets, and we see that a presentation by generators and relations is formally isomorphic to a propositional geometric theory with signature and axioms: the theory acts as a presentation for the Lindenbaum algebra.

We can simplify this. First, frame distributivity allows us to express each side of a relation (axiom) as a join of finite meets; then, by the meaning of joins, the relation can be replaced by a set of relations of the form
\[
\text{finite meet in } G \leq \text{join of finite meets in }G
\text{.}
\]
We can introduce notation for this if the relation is $r\in R$:
\[
\unmeet \lambda(r) \leq \unjoin_{d\in D(r)} \unmeet \rho(d)
\text{.}
\]
($\lambda$ and $\rho$ stand for left and right.)

Putting $D=\sum_{r\in R} D(r)$ (D for \emph{disjuncts}), with projection $\pi\colon D\to R$, the whole structure can be summarized~\cite{PPExp} as a \emph{GRD-system:}
\[
\begin{tikzcd}
	& D
	\ar[d, "\pi"]
	\ar[dl, "\rho"']
	\\
	\fin G
	& R
	\ar[l, "\lambda"]
\end{tikzcd}
\]
(Note that the diagram is \emph{not} required to commute.)

\begin{definition}
	\label{def:ptT}
	Let $\Sigma=(G,R,D, \lambda, \pi, \rho)$ be a GRD-system.
	Then its space of points $\pt{\Sigma}$ corresponds to the propositional theory with signature $G$, and axioms
	\[
	\unmeet_{g\in \lambda(r)}g \vdash
	\unjoin_{\pi(d)=r}\unmeet_{g\in \rho(d)}g
	\quad (r\in R)
	\text{.}
	\]
\end{definition}

This is perhaps the most straightforward structure to take the syntax out of the notion of propositional geometric theory.
Both formal topologies and further examples in \cite{PPExp} use extra structure to simplify certain calculations.

\subsection{Sites}
\label{sec:Sites}

\newcommand{\Flat}{\mathsf{Flat}}
\newcommand{\CtsFlat}{\mathsf{CtsFlat}}

We now move on to predicate theories.
For these, in general, frames are not enough and toposes must be used instead.

A central fact underlying much topos theory is that there is a standardized way to present geometric theories, using \emph{sites} $(\catC,J)$.
Here $\catC$ is a small category, and $J$ a \emph{coverage,}
a set of \emph{covers}.
Each cover, covering $A$, is a set of morphisms in $\catC$ with codomain $A$.

Sites are also required to satisfy additional properties in order to make the theory of sheaves work well.
One such, discussed in~\cite[A2.1.9]{Elephant1}, is
that if $C=\{f_i\colon A_i\to A\mid i\in I\}$ covers $A$,
and $g\colon A'\to A$,
then there some $C'=\{h_{i'}\colon A'_{i'}\to A'\mid i'\in I'\}$ covering $A'$ such that each $h_{i'};g$ factors through some $f_i$.

However, even without such a site property, the pair $(\catC,J)$ presents a geometric theory of \emph{flat, continuous functors}
(see~\cite{MacMoer,LocTopSp}), and we shall write $\pt{\catC,J}=\CtsFlat(\catC,J)$ for the corresponding space.

First, given just $\catC$, we have a space $\Flat(\catC)$ of flat functors on $\catC$.
Its theory has a sort $X_A$ for each object $A$ of $\catC$,
and a function symbol $u_f\colon X_A \to X_B$ for each morphism $f\colon A\to B$.

Then the axioms express (i) that the action $u$ respects identities and composition,
\begin{align*}
  & \top \turnstile{x\oftype X_A} u_{\Id_A}(x) = x
    && A \in \obj\catC
  \\
  & \top \turnstile{x\oftype X_A} u_g(u_f(x)) = u_{f;g}(x)
    && \begin{tikzcd}[ampersand replacement = \&]
    	A \ar[r, "f"] \& B \ar[r, "g"] \& C
    \end{tikzcd}
\end{align*}
and (ii) flatness
\begin{align*}
  & \top \turnstile{} \unjoin_{A\in \obj\catC} \Exists {x\oftype X_A} \top
  \\
  & \top \turnstile{x\oftype X_A,y\oftype X_B}
    \unjoin_{\begin{tikzcd}
    		[ampersand replacement = \&, cramped, row sep = -0.8em, column sep = scriptsize]
    		\&A \\ C \ar[ur, "f"] \ar[dr, "g"']\\ \& B
    	\end{tikzcd}}
      \Exists {z\oftype X_C} (x=u_f(z) \wedge y = u_g(z))
    && A, B \in \obj\catC
  \\
  & u_f(x)=u_g(x) \turnstile{x\oftype X_A}
    \unjoin_{\begin{array}{c}
    		h\colon C\to A \\ h;f = h;g
    \end{array}}
      \Exists {y\oftype X_C} x= u_h(y)
    && \begin{tikzcd}[ampersand replacement=\&, cramped]
    	A \ar[r, shift left, "f"] \ar[r, shift right, "g"'] \& B
    \end{tikzcd}
\end{align*}

Next, with the coverage $J$, we extend the theory for $\Flat(\catC)$ with more axioms to give a subspace $\CtsFlat(\catC,J)$.
For each covering family in $J$, say $\{f_i\colon A_i\to B\mid i\in I\}$, we add an axiom
\[
\top \turnstile{y\oftype X_B}
\unjoin_{i\in I}\Exists {x\oftype X_{A_i}} u_{f_i}(x)=y
\text{.}
\]
This is our $\pt{\catC,J}$.

\section{Formal duality}
\label{sec:FormDual}
The aim in this Section is to explain why the rather circular discussion in the Glossary has a well-founded model in topos theory.
Although we shall stay with the phrase \glemph{Giraud frame}, it means roughly the same (at least up to dimension 2) as the \glemph{logos} of~\cite{AnelJoyal:Topologie}.
Their title ``Topo-logie'' expresses the idea of \glemph{formal duality}, that space (topos) and algebra (logos) can be different aspects of the same structure.

As mentioned in Section~\ref{sec:Intro}, we shall be interested in seeking models of a geometric theory $\thT$ in categories $\catE$ other than $\baseS$.
Following topos theory, we take $\catE$ to be a \glemph{Giraud frame}, essentially the same as a \glemph{Grothendieck topos} or a \glemph{logos}.%
\footnote{
	Strictly speaking, if $\thT$ is in purely logical form then all we need is for a category to be a \emph{geometric} category~\cite[A1.4.18]{Elephant1},
	with infinitary joins of subobjects for the infinitary disjunctions.
	In practice that is not sufficiently well behaved to give us good internal mathematics.
	In particular, it is not in general balanced (mono epis need not be isomorphisms), and so, logically, it does not necessarily have unique choice. This is a severe limitation to the internal mathematics, as it means functions cannot be defined by their graphs.
	Also they do not necessarily have colimits,
	so they do not interpret the sort constructors included in the more liberal forms of geometric theory.
}
One can prove that Giraud's categorical conditions are good for modelling the rules of geometric logic (see, \eg~\cite[D1]{Elephant2}).

The appropriate morphisms of Giraud frames are the functors that preserve the geometric structure of colimits and finite limits,
and we thus get a category $\GFr$ of Giraud frames and geometric functors.
In fact, it is a 2-category, with the 2-cells being natural transformations.

It can be proved that any Giraud frame morphism has a right adjoint, so (up to isomorphism of the right adjoints) the Giraud frame morphisms are dual to geometric morphisms.

\subsection{Using geometric theories to present Giraud frames}
\label{sec:PresGFr}

Recall the definition of \glemph{classifying topos}.
We express it in terms of Giraud frame morphisms instead of geometric morphisms, in order to emphasize the analogy with \glemph{classifying categories}. 

\begin{definition}
\label{def:GFrPres}
	Let $\thT$ be a geometric theory, and $\catE$ a Giraud frame. Then $\catE$ is \emph{a classifying topos for} $\thT$ if
	\begin{enumerate}
    \item
        $\catE$ is equipped with a \emph{generic} model $\gamma_\thT$ of $\thT$, and
    \item
        For every Giraud frame $\catF$, the functor
        \[
        \GFr(\catE, \catF) \to \thT(\catF)
        \text{,}\quad
        F\mapsto F(\gamma_\thT)
        \text{,}
        \]
        is one half of an equivalence of categories.
        Here $\thT(\catF)$ is the category of models of $\thT$ in $\catF$.
	\end{enumerate}
\end{definition}

Note that it is 2-categorical. (More correctly, since it asks for an equivalence of categories and not an isomorphism, it is bicategorical.)
Not only do models of $\thT$ correspond to Giraud frame morphisms, but also homomorphisms of models correspond to 2-cells in $\GFr$.

What the definition amounts to is that $\thT$ is being used as generators (the signature) and relations (the axioms) by which $\catE$ is \emph{presented}.
The generic model is the ``injection of generators'' into $\catE$.
Then the universal property is that any model $M$ in $\catF$ extends uniquely (up to isomorphism) to a Giraud frame morphism on the whole of $\catE$ that restricts on $\gamma_{\thT}$ to a transformation to $M$.

The pair $(\catE, \gamma_\thT)$ is defined uniquely up to equivalence, and, if it exists, we generally write $\baseS[\thT]$ for $\catE$.
The notation can be read as ``$\baseS$ freely augmented with a generic model of $\thT$'', analogous to polynomial rings $R[X]$, although we shall see that it can also be read as ``$\baseS$heaves over the space $[\thT]$''.

We can summarize this with a slogan:
\begin{quote}
	$\baseS[\thT]$ is the geometric mathematics (relative to $\baseS$) freely generated by a model of $\thT$.
\end{quote}

\subsection{Points = models}
\label{sec:PointsWhy}

We now explain why the discussions of Sections~\ref{sec:Sites} and \ref{sec:PresGFr} lead directly via \glemph{formal duality} to an understanding of $\GFr^{op}$ as a category of spaces.
In fact this is a very general procedure, valid for any logic for which theories can present \glemph{classifying categories}.
However, for infinitary logics, the construction is more delicate and may be impossible.
For geometric theories, we are able to construct \glemph{classifying toposes} because (i) every theory is equivalent to one presented by a site (Section~\ref{sec:Sites}), and (ii) for a site, the theory is classified by the category of sheaves.

Of course, in geometric logic, it is the infinities that play the decisive role in giving it a topological flavour: in the match between the disjunctions and finite conjunctions, logically, and the unions and finite intersections of opens, topologically. (Nonetheless, it is surprising how much of that can be captured finitarily using \glemph{AUs}.) 

Now consider a geometric theory $\thT$ with classifier $\baseS[\thT]$, and a Giraud frame $\catE$.
By definition, the models of $\thT$ in $\catE$ are equivalent to the Giraud frame morphisms $\baseS[\thT]\to\catE$.
But if we reverse the arrows and work in $\GFr^{op}$, these are just the (generalized) \glemph{points} of $\baseS[\thT]$ at stage $\catE$.
Thus $\GFr^{op}$, with $\thT$ represented by $\baseS[\thT]$, implements the space part of Definition~\ref{def:ptfree}: the points of $\baseS[\thT]$ are the models of $\thT$.

By \glemph{Giraud's Theorem} we know, further, that $\catE$ is of the form $\baseS[\thT']$.
Then a morphism in $\GFr^{op}$ from $\baseS[\thT]$ to $\baseS[\thT']$ is a model of $\thT'$ in $\baseS[\thT]$; and that is a model of $\thT'$ constructed geometrically from the generic model of $\thT$.
With a similar reasoning for 2-cells, we see that thus $\GFr^{op}$ also implements maps from Definition~\ref{def:map}.

Thus we can understand the \glemph{space of models} $[\thT]$ as the classifying topos $\baseS[\thT]$, treated as an object of $\GFr^{op}$.
If $f\colon [\thT]\to[\thT']$, then we shall write $f^\ast\colon\baseS[\thT']\to \baseS[\thT]$ for the corresponding Giraud frame morphism, and $f_\ast$ for its right adjoint.
Thus $(f^\ast \dashv f_\ast)$ is a geometric morphism \emph{from} (in the usual terminology) $\baseS[\thT]$ \emph{to} $\baseS[\thT']$.

Note that the transformation $x\mapsto f(x)$, defined in Definition~\ref{def:map}, is a transformation of \emph{points,} and must not be considered as either $f^\ast$ or $f_\ast$. 
On a $\thT$-model $M$ in $\catE$, it can equivalently be understood in three different ways.
First, as a geometric construction, it can be applied in $\catE$ to $M$ to give $f(M)$.
Second, $M$ corresponds to a Giraud frame morphism $G\colon\baseS[\thT]\to\catE$. This composes with $f^\ast\colon\baseS[\thT']\to \baseS[\thT]$ to give a morphism $\baseS[\thT']\to\catE$, and hence a model of $\thT'$ in $\catE$.
Third, $f(\gamma_\thT)$ is a model of $\thT'$ in $\baseS[\thT]$, and applying $G$ to it gives a model of $\thT'$ in $\catE$.

\section{Syntacticization: translating mathematics into logic}
\label{sec:Syntactn}
In this section we show how to translate from an ordinary mathematics using sets and functions into ingredients of a geometric theory.

On the face of it, it looks like a demonstration of the power of geometric logic, with its arbitrary disjunctions:
we can use geometric structure and axioms to constrain a sort so that, for instance, in any model it is carried by a set isomorphic to a given set $A$.
This is impossible in general for finitary first-order logics.

One might conclude that the techniques provide a useful tool for translating mathematics into the syntax of geometric logic.
This would be a mistake, however, as performing the translation is tedious and not an essential part of the mathematics.
A better conclusion is that, if we are interested in geometric theories, then we might as well engineer their syntax to allow the use of such ``constant'' sorts and functions, and geometric constructions.
We shall say more about this in Section~\ref{sec:GeoCons}.
In fact it is a key tool in the ability to relativize: theories can be built up in stages by incorporating new ingredients.

We shall refer to this process as \glemph{syntacticization} of geometric constructions, whether it is done using infinitary disjunctions (this Section), or by liberalizing the syntax (Section~\ref{sec:GeoCons}).
\subsection{Syntacticizing sets and functions} \label{sec:SyntactnSetFn}
First, we look at how to encode a set $A$ into geometric structure and axioms.
What we do is to introduce a sort $\sigma_A$ with constants $\kappa_a$, one for each $a\in A$.
We then want to axiomatize that every $x$ corresponds to a unique $a$.
For existence, we can obviously use a disjunction $\unjoin_{a\in A}x=\kappa_a$, and it is infinitary if $A$ is infinite.
Less obviously, we also need the general disjunctions for uniqueness, and this is because we do not assume excluded middle.
In Proposition~\ref{prop:setToSort} this appears as
$\unjoin\{\top\mid a=a'\}$. Classically this would appear as $\bot$, but under a side condition $a\neq a'$.

In the following result, note that, although we have stated it for sets and functions ``in $\baseS$'', it still works if we replace $\baseS$ by a Giraud frame $\catE$ over it.

\begin{proposition}
\label{prop:setToSort}
  Sets $A$ (objects in $\baseS$) can be incorporated into geometric theories as sorts $\sigma_A$ so that in any model (in a Giraud frame $p\colon\catF\to\baseS$), the interpretation of $\sigma_A$ is canonically isomorphic to $p^\ast A$.
  
  Likewise, functions $u\colon A\to B$ can be incorporated as function symbols $f_u$.
\end{proposition}
\begin{proof}
  Introduce an $A$-indexed family of constants $\kappa_a$ on $\sigma_A$ with axioms
  \begin{align*}
  	& \top \turnstile{x\oftype\sigma_A}\unjoin_{a\in A}x=\kappa_a
  	\\
  	& \kappa_a=\kappa_{a'} \turnstile{x\oftype\sigma_A}
  	\unjoin\{\top\mid a=a'\}
  	&
  	&(a,a'\in A)
  \end{align*}
  They say that each $x$ corresponds to exactly one $a$.
  In $\baseS$, $A$ is an internal coproduct of an $A$-indexed family of copies of 1, and this is preserved by $p^\ast$.%
  \footnote{
    We are concealing a lot here.
    The mechanism required to discuss $\baseS$-indexed coproducts in other Giraud frames is quite elaborate, involving indexed categories~\cite[B1.4.8]{Elephant1}.
  }
  In the model, each $\kappa_a$ gives a global element of $\sigma_A$, and these combine to give a map from $p^\ast A$. The axioms force it to be an isomorphism.
  
  In any $\catF$ we can define an essentially unique model with $\sigma_A$ interpreted as $p^\ast A$, and with $\kappa_a$ as the image under $p^\ast$ of $\{a\}$, so we see that the theory is Morita equivalent to 1.

  Now suppose we have a morphism $u\colon A\to B$ in $\baseS$. We first form a geometric theory with $\sigma_A$ and $\kappa_a$ as above, and similarly $\sigma_B$ and $\lambda_b$ for $B$. In the (essentially unique) model, we have a morphism between the interpretations of $\sigma_A$ and $\sigma_B$, with graph given by the formula in context
  \[
    (x\oftype\sigma_A,y\oftype\sigma_B\mid
      \unjoin_{a\in A} (x=\kappa_a\wedge y=\lambda_{u(a)})
  \]
  Modulo the canonical isomorphisms, it agrees with $p^\ast(u)$.
  
  In the theory, this morphism can be expressed by adjoining a function symbol $f_u\colon \sigma_A\to\sigma_B$ with axioms
  \begin{align*}
	& \top \turnstile{} f_u(\kappa_a) = \lambda_{u(a)}
	&
	& (a\in A)
  \end{align*} 
\end{proof}
Note how the constants $\kappa_a$ make an explicit translation between the logic of the geometric theory ($x\oftype \sigma_A$, $\exists {x\oftype\sigma_A}$, $x=x'$)
and the internal mathematics of $\baseS$
($a\in A$, $\unjoin_{a\in A}$, $\unjoin\{\top\mid a=a'\}$).

\subsection{Translating geometric constructions}
\label{sec:SyntactnGeomCons}

Here we outline the geometric ingredients needed to encode geometric constructions: colimits and finite limits.
Again, it would be a mistake to think of these as useful tools for turning mathematics into first order logic.
Rather, it is an indication that we should aim for a generalized geometric syntax that encompasses these constructions.

\subsubsection{Pullbacks}

Suppose we have a cospan $f_i\colon\sigma_i \to \sigma_3$ ($i=1,2$).
We can force a sort $\sigma$ to give the pullback using projections $p_i\colon\sigma\to\sigma_3$ and axioms
\begin{align*}
  	& \top  \turnstile{x\oftype\sigma} f_1(p_1(x))=f_2(p_2(x))
	\\
	& p_1(x)=p_1(x') \wedge p_2(x)=p_2(x')
	   \turnstile{x,x'\oftype\sigma} x=x'
	\\
	& f_1(y_1) = f_2(y_2)
	  \turnstile{y_i\oftype\sigma_i}
	  \Exists {x\oftype\sigma} (y_1=p_1(x) \wedge y_2=p_2(x))
\end{align*}

If we have a span $g_i\colon \tau \to \sigma_i$ with $f_1\circ g_1 = f_2\circ g_2$, then we can define the fillin $g\colon\tau\to\sigma$ as a logical formula for the graph,
\[
  \gamma(z,x) := g_1(z) = p_1(x) \wedge g_2(z) = p_2(x)
  \text{.}
\]
This is provably functional, and unique as fillin.

The terminal object is easily translated into logic, and so we have all finite limits.

Note that no disjunctions are needed for these.
(See~\cite{PHLCC} for an account of the finite limits using a minimal amount of partial logic.)

\subsubsection{Coproducts}

Suppose we have a family of sorts $\sigma_i$ ($i\in I$),
and we want to introduce a sort $\sigma$ and constrain it to be the coproduct $\coprod_i \sigma_i$.
Let us also introduce the coproduct injections $\inj_i\colon\sigma_i\to\sigma$.
We know that coproducts in any Giraud frame are disjoint;
conversely, if the $\inj_i$s are disjoint and cover $\sigma$, then $\sigma$ (in any model) is the coproduct of the $\sigma_i$s. We can force this with axioms
  \begin{align*}
	& \top \turnstile{x\oftype\sigma}\unjoin_{i\in I}\Exists {y\oftype\sigma_i} x=\inj_i(y)
	\\
	& \inj_i(y)=\inj_{i'}(y') \turnstile{y\oftype\sigma_i, y'\oftype \sigma_{i'}}
	\unjoin\{y=y'\mid i=i'\}
	&
	&(i,i'\in I)
\end{align*}
In fact, Proposition~\ref{prop:setToSort} was a special case of this, for copowers of 1.

\subsubsection{Natural numbers}

The natural number object $N$ is also a constant set, but it can be dealt with in a more structured way using constant $0$ and successor function $s\colon N\to N$, with axioms
\begin{align*}
  & s(x)=0 \turnstile{x\oftype N} \bot
  \\
  & s(x)=s(x') \turnstile{x,x'\oftype N} x=x'
  \\
  & \top \turnstile{x\oftype N}
    \unjoin_{n\in\Nat} x = s^n(0)
\end{align*}
The term $s^n(0)$ is not strictly a term in first-order logic, but an indication of how can be built recursively in $n$.
\subsubsection{List objects}

Suppose $\sigma$ is a sort. We can then define the list sort $L=\List\sigma$ using constant $\epsilon$ (empty list), constructor map
$\cons\colon\sigma\times L \to L$, and axioms
\begin{align*}
  & \cons(x,l)=\epsilon \turnstile{x\oftype\sigma, l\oftype L} \bot
  \\
  & \cons(x,l)=\cons(x',l') \turnstile{x,x'\oftype \sigma,l,l'\oftype L} x=x' \wedge l=l'
  \\
  & \top \turnstile{l\oftype L}
\unjoin_{n\in\Nat} \Exists {x_1,\ldots,x_n\oftype\sigma} l = \cons(x_1\ldots\cons(x_n,\epsilon)\ldots)  
\end{align*}
Much as with the natural number object, the axioms allow an element $l\oftype\List\sigma$ to be decomposed uniquely as a finite list of elements of $\sigma$,
and then the fillin function for the universal property (of parametrized list object) can be expressed, via its graph, as a disjunction showing the result for each explicit list.

\subsubsection{Coequalizers}

Coequalizers of equivalence relations (as we get from pretopos structure) are relatively easy.
Suppose we have an equivalence relation $\begin{tikzcd}
	\sigma_1
	  \ar[r, shift left, "p_1"]
	  \ar[r, shift right, "p_2"']
	& \sigma_2
\end{tikzcd}$.
Then we can structure the quotient $\sigma$ using a function $p\colon \sigma_2 \to \sigma$ and axioms
\begin{align*}
	& \top \turnstile{e\oftype \sigma_1} p(p_1(e))=p(p_2(e))
	\\
	& \top \turnstile{x\oftype \sigma}
	  \Exists {y\oftype \sigma_2} x=p(y)
	\\
	& p(y)=p(y') \turnstile{y,y'\oftype\sigma_2}
	  \Exists {e\oftype\sigma_1} (y=p_1(e) \wedge y'=p_2(e))
\end{align*}

The more general coequalizer is more complicated, but can be constructed using $\List\sigma_1$ to construct the transitive equivalence relation generated by $\sigma_1$.

\subsubsection{Finite powersets}

The case of the finite powerset $\fin$ is dealt with in some detail in~\cite{TopCat}, which also shows how to incorporate \emph{finitely bounded} universal quantification into geometric logic, translating $\Forall {x\in S}\phi$ into
\[
  \unjoin_{n=0}^\infty \Exists {x_1,\ldots,x_n}(S=\{x_1,\ldots x_n\}\wedge \unmeet_{i=1}^n \phi[x_i/x])
  \text{.}
\]

\section{Geometric constructions}
\label{sec:GeoCons}
Section~\ref{sec:Syntactn} showed how to translate mathematics, or at least the geometric part of it, into geometric theory ingredients.
However, as we suggested there, that may not be such a good idea, as the syntax of first order logic is not that convenient for doing mathematics.
Here we look at a formalism that is more in the nature of type theory than of logic.

In~\cite{LocTopSp}, an informal idea of ``geometric type theory'' was described, in which sort constructors would be allowed provided the constructions were geometric (preserved by inverse image functors of geometric morphisms).

The intuition is that certain constructions can be characterized, uniquely up to isomorphism, by geometric structure and axioms.
In Section~\ref{sec:Syntactn} we saw this for ``constant'' constructs, objects and morphisms of $\baseS$, 
and further constructions can be justified on a case-by-case basis.

Encoding the constructions into geometric structure and axioms is tedious, and mathematically unnatural, and so we are led to seek a ``geometric type theory'' in which they can be used directly as sort constructors in the logic. Already in~\cite{TopCat,SFP}, this had been used to justify finitely bounded universal quantification as part of geometric logic, with constructed sorts $\fin X$ for Kuratowski finite powersets (free semilattices).
The principle was illustrated in~\cite{LocTopSp} with a first-order theory of Dedekind sections, using the rationals $\Qrat$ as a geometrically constructed sort.

The question then is how we might formalize this idea of geometric sort constructors, bearing in mind that we have to construct not only the sorts, but also morphisms; and that the constructions too may depend on morphism. As an example, consider constructing and exploiting colimits and finite limits, the primaeval geometric constructions. 

We outline here (Sections~\ref{sec:ElephantTh} and~ \ref{sec:ArithTh}) two extreme solutions. Each involves building a geometric theory up in steps of adjoining ingredients, in a way that includes using geometric constructions.
The first, of Johnstone~\cite[B4.2.7]{Elephant1}, allows for all constructions geometric over a fixed base $\baseS$.
The second, of Vickers~\cite{Vickers:AUSk} (but anticipated in~\cite{TopCat}), uses \glemph{arithmetic universes} to capture a fragment of the geometric constructions valid over all choices of \glemph{elementary topos} (with nno) $\baseS$.

\subsection{Elephant theories}
\label{sec:ElephantTh}
The two volumes~\cite{Elephant1,Elephant2} of Johnstone's Elephant define geometric theories twice. The second (D1.1.6) is in the conventional style of a first-order theory. It is essentially syntactic, though there are issues with defining infinite disjunctions, which have to refer back to an ambient set theory.

The first (B4.2.7) is rather different, and deals explicitly with geometric theories relative to a base $\baseS$. It start by defining ``theory'' $\thT$ in general in terms of the models. It is known for geometric theories that there may be insufficient models in $\baseS$, so a category $\thT(\catE)$ of models has to be specified for \emph{every} bounded $\baseS$-topos $\catE$ (\glemph{Giraud frame} relative to $\baseS$).
If $\thT$ is geometric, then for a geometric morphism $f\colon\catE\to\catF$ the inverse image functor $f^\ast$ should map models in $\catF$ to models in $\catE$.
This is formalized by saying that $\thT$ should be an indexed category over $\BTop/\baseS$.

Following~\cite{Vickers:AUClTop}, where there is further discussion of such theories, we call them \emph{elephant theories}.%
\footnote{They are very large.}

A \emph{geometric construct} in $\thT$ is then defined~(B4.2.5) as a $\BTop/\baseS$-indexed functor $F\colon\thT\to\thobj$, where $\thobj$ is the theory of \glemph{sets}, $\thobj(\catE)=\catE$.
As Johnstone says,
``Informally, $F$ may be thought of as something which `constructs' an object $F(M)$ of $\catE$ from a $\thT$-model $M$ in an $\baseS$-topos $\catE$, and a morphism $F(M)\to F(N)$ from each morphism of $\thT$-models $M\to N$, in a manner which is compatible with composition of $\thT$-model morphisms and with the functors $f^\ast$ between $\thT$-models in different $\baseS$-toposes.''

Notice how it is the compatibility with functors $f^\ast$ that gives the geometricity as preservation by inverse image functors.

The geometric constructs are essentially the objects of the classifying topos, the \glemph{sheaves} over the space of models.
They include both ``constant'' constructs (objects of $\baseS$, cf. Section~\ref{sec:SyntactnSetFn}) and more general geometric constructions (cf. Section~\ref{sec:SyntactnGeomCons}).

We may consider extending $\thT$ to a theory $\thT'$ whose models are pairs $(M,\Gamma)$, where $M$ is a model of $\thT$ and $\Gamma=F(M)$.
However, we have made no essential change, as the new elephant theory is isomorphic to the old one.

\begin{definition} \label{def:geoThEleph}
A \emph{geometric theory} according to~\cite[B4.2.7]{Elephant2} is built up in finitely many steps of the following kinds, each extending a theory $\thT$ to give a theory $\thT'$ whose models are pairs $(M, \Gamma)$ where $M$ is a model of $\thT$ and $\Gamma$ is some other structure, in general dependent on $M$.

\begin{itemize}
\item
  \emph{(Primitive sort)}
  $\Gamma$ is a set.
  
  Syntactically, this would adjoin a new sort.%
  \footnote{
    Actually, \cite{Elephant2} requires steps of this kind all to be done at the start. That makes no essential difference, as these steps do not interact with the others, and so can be moved earlier.
  }
\item
  \emph{(Simple functional extension)}
  Suppose $F$ and $F'$ are two geometric constructions. We take $\Gamma$ to be a morphism $f\colon F(M) \to F'(M)$.
  
  Syntactically, this would correspond to incorporating $F$ and $F'$ as new sorts, constrained by the definitions of them as geometric constructions, and then adjoining a function symbol between them.
\item
  \emph{(Simple geometric quotient)}
  Suppose $F$ and $F'$ are two geometric constructions, and $f\colon F \to F'$ is a natural transformation.
  We take $\Gamma$ to be an inverse of $f_M\colon F(M) \to F'(M)$.
  
  Syntactically, this would correspond to adjoining a function $g\colon F'(M) \to F(M)$, with axioms to require its composites with $f_M$ to be identities.
  Thus $\top \vdash_{x\oftype F(M)} g(f_M(x))=x$ and similarly for $y\oftype F'(M)$.
  
  Note that any geometric sequent can be worked into this form, as satisfying $\phi \vdash_x \psi$ is equivalent to having an inverse for the inclusion $\phi\wedge\psi \to \phi$.
\end{itemize}
\end{definition}

Notice how the definition implicitly allows geometric constructs to be incorporated into theories as what would, syntactically, be sort constructors, though semantically, as elephant theories, there is no essential change.
Implicitly, this allows us to \glemph{syntacticize} both constants and constructions.

\subsection{Arithmetic theories}
\label{sec:ArithTh}
We are seeking to describe as large as possible a class of geometric constructions, allowing a range of possible $\baseS$s.
The elephant approach does this separately for each $\baseS$. A different approach is to seek a single class of constructions that work for every $\baseS$ (with nno), and this is done in~\cite{Vickers:AUSk} using the constructions of \glemph{arithmetic universes}.%
\footnote{
  We conjecture that this captures exactly the geometric constructions that can be implemented in every elementary topos with nno.
}
The sorts constructed are the finite limits, \emph{finite} colimits, and parametrized list objects, together with the derivable morphisms and equations between morphisms.
These are all expressed in sketch form, and the \emph{contexts} defined there, which we shall think of as \glemph{arithmetic theories}, are built up stepwise in the manner of Definition~\ref{def:geoThEleph}.
The big difference is that it does not rely on the semantic notion of geometric construct in terms of elephant theories, but instead has explicit steps for \glemph{syntacticizing} some particular ones: initial and terminal objects, pullbacks, pushouts and list objects, together with the associated morphisms and equalities of morphisms.
These are set up as ``equivalence extension steps'', to acknowledge the fact that semantically they make no difference.

(Note that we have no $\baseS$ and cannot syntacticize constants other than those that occur already in the initial AU, constructed using finite limits, finite colimits, and list objects.)

Other equivalence extension steps account for the pretopos nature of the semantics, and for the fact that the sketch-based syntax uses graphs instead of categories (this is in order to keep the sketches finite).

After that, the steps for more general extensions use primitive sorts and functional extensions, as in Definition~\ref{def:geoThEleph}, and steps to impose equalities between morphisms.

\cite{Vickers:AUClTop} then shows how AU techniques can give results for toposes in a base-free way.

Note that, despite the restriction to finite colimits, the arithmetic theories are definitely stronger than \glemph{coherent theories}. In finitary logic, the L\"owenheim-Skolem Theorem shows that no set of axioms can suffice to force a sort to be isomorphic to the natural numbers, so, in a first order theory, infinitary disjunctions are required if we are to do that. Thus, in admitting the natural numbers as a sort constructor, we must be going non-coherent.

Having made that step, it is surprising how expressive the arithmetic theories are. For example, the real line $\Reals$ can be expressed in arithmetic form. Nonetheless, it is at present a huge open question how much of the applications of topos theory can be made arithmetic.

\section{Relative theories}
\label{sec:RelTh}
Suppose we have a family of spaces $\pt{A}$ that can be defined for structures $A$ of a certain kind.
If $A$ lives in a Giraud frame $\catE$, then $\pt{A}$ will correspond to an internal theory, with an internal signature and internal set of axioms. Then by changing our base from $\baseS$ to $\catE$, we can generate a classifying topos $\catF = \catE\pt{A}$, a bounded $\catE$-topos.
We have \emph{relativized} the process to $\catE$ (from $\baseS$).

Let us look at the localic example of $A$ being a frame. If $\catE$ is localic over $\baseS$, then the frame for $\catF$ can be calculated~\cite{JoyalTier} as $p_\ast(q_\ast(\Omega_{\catF}))$ using the geometric morphisms
\[
  \begin{tikzcd}
  	\catF=\catE[\pt{A}]
  	  \ar[r, shift left=1.5ex, "q_\ast"]
  	  \ar[r, phantom, "\top" description]
  	& \catE
  	  \ar[r, shift left=1.5ex, "p_\ast"]
      \ar[r, phantom, "\top" description]
      \ar[l, shift left=1.5ex, "q^\ast"]
    & \baseS
      \ar[l, shift left=1.5ex, "p^\ast"]
  \end{tikzcd}
\] 

Since bounded geometric morphisms are closed under composition, $\catF$ is also bounded over $\baseS$, and so one might well ask what it classifies.
In a sense, that's easy -- it classifies completely prime filters over $p_\ast(q_\ast(\Omega_{\catF}))$.
However, the calculations there of $\Omega_{\catF}$,
$p_\ast$, and $q_\ast$ are non-geometric,
and seem to depend on \emph{constructing} the Giraud frames.

What we shall show is that it is possible to circumvent that, and work entirely with the geometric descriptions.
For frames, the problem geometrically is that the theory of frames is not geometric -- see Section~\ref{sec:Frames}.
When we referred above to structures ``of a certain kind'' that give rise to theories, our central technique is to use structures that (unlike frames) can be described geometrically.
For example, if $A$ is a GRD-system then we have defined $\pt{A}$.
This is discussed in some detail in~\cite{PPExp}.

Let us first review the general question where structures $A$ (in $\baseS$) ``of some suitable kind'' (which we now suppose can be described geometrically) give rise to spaces $\pt{A}$.
The theory for $\pt{A}$ can equivalently be expressed in a form in which the ingredients of $A$ have been syntacticized as constants.
This uses either the explicit geometric axiomatization as in Section~\ref{sec:Syntactn}, or, and more conveniently, the type-theoretic ideas of Section~\ref{sec:GeoCons}.

The trick then is to express the theory ingredients of $\pt{A}$ in terms of the syntacticized $A$ rather than the original external $A$. We shall see this in Proposition~\ref{prop:ptT}.
After that, it doesn't matter whether $A$ was a constant or declared as part of a theory, and this allows us to built $\pt{A}$ as an extension of a theory that $\catE$ classifies.

\subsection{Example: GRD-systems}
\label{sec:GRDex}
Suppose $\Sigma$ is a GRD-system as in Definition~\ref{def:ptT}.
Then the theory for $\pt{\Sigma}$ is propositional,
with propositions indexed by $G$ and axioms indexed by $R$ and expressed using the other ingredients of $\Sigma$.
What we shall show now is how to transform that into a first-order theory that provides the key to relativizing.

Our first step is to construct sorts ``out of nothing'' to syntacticize $\Sigma$, replicating it in the signature.
For clarity, let us call its sorts $\sigma_G$, $\sigma_R$, $\sigma_D$ and $\sigma_{\fin G}$. (In practice we could just call them $G$, $R$, $D$ and $\fin G$.)
Each can then be equipped with constants for the elements, for instance $\kappa_g\colon 1\to \sigma_G$ for $g\in G$.

In the style of Section~\ref{sec:SyntactnSetFn}, we should then equip them with axioms to ensure that in any model (in a bounded $\baseS$-topos $p\colon \catE\to\baseS$) the interpretation of $\sigma_G$ is isomorphic to $p^{\ast}(G)$, and so on.

Obviously those explicit translations would be tedious if we actually did them. Moreover, they are mathematically inessential, being just a concession to the strict syntax of first order logic. It would therefore be better to rely on the knowledge that they can be done, and treat the external language as something that can be imported directly into geometric theories. In this style, we simply import $G$, $R$ and so on as sorts -- and let's not be fussy about calling them $\sigma_G$ and so on.

Once we have done this, we can rewrite the theory in Definition~\ref{def:ptfree} in a way that turns out to be highly beneficial.

\begin{proposition}
	\label{prop:ptT}
	The theory for $\pt{\Sigma}$ in Definition~\ref{def:ptT} can equivalently be presented as follows.
	
	The signature has a unary predicate $F$ on $G$.
	There is one axiom,
	\[
	\Forall {g\in\lambda(r)} F(g) \turnstile{r\oftype R}
	\Exists {d\oftype D}(\pi(d)=r \wedge \Forall {g\in \rho(d)}F(g))
	\text{.}
	\]
\end{proposition}
\begin{proof}
	See~\cite{PPExp}.
\end{proof}
In practice it may be more mathematically natural to treat $F$ as a subset of $G$ and write $g\in F$ (or, type theoretically, $g\oftype F$) instead of $F(g)$.
We don't want to be rigid about syntax -- the aim is to be mathematically natural.

If, as anticipated in Section~\ref{sec:FrPres}, we introduce some dependent type theory with
$r\oftype R\turnstile{}D(r)\oftype\Set$ and $D=\sum_{r\oftype R} D(r)$,
then the axiom of Proposition~\ref{prop:ptT} becomes
\[
\Forall {g\in\lambda(r)} g\oftype F \turnstile{r\oftype R}
\Exists {d\oftype D(r)} \Forall {g\in \rho(d)}g\oftype F
\text{.}
\]

So far, $G$ etc. are still just ``constant'' sets, imported from $\baseS$.%
\footnote{
  At least, we are discussing them as if they were.
  In actual fact, neither you nor I have a clue what sets they are.
  They are entirely hypothetical and our discussion awaits actual sets to substitute for the formal symbols.
  In reality, they are already behaving like ``declared sorts''.
}
However, the formulation in Proposition~\ref{prop:ptT} does not care whether they are constants or declared sorts, declared in a subtheory for a space $\GRD$ of GRD-systems,
which is then extended to a theory whose models are pairs $(\Sigma,F)$ where $\Sigma$ is a GRD-system and $F$ a point of $\pt{\Sigma}$.
We shall write this in a notation like that for \emph{contexts} in dependent type theory: $(\Sigma\oftype\GRD, F\oftype\pt{\Sigma})$.

This implicit translation between the external logic and the internal mathematics has thus led us to the important idea of \emph{relativization}, with geometric theories relative to a topos.

For example, suppose we have a map $\Sigma\colon X\to \GRD$, in other words a GRD-system in $\baseS X$.
Then we can make an extended theory $(x\oftype X, F\oftype\pt{\Sigma(x)})$.
This follows the stepwise construction of geometric theories as in Section~\ref{sec:GeoCons}.
Iterating this, we can easily get a tower of relativizations, with internal at one level being external at the next.

\subsection{Sites} \label{sec:RelSites}

The relativization process illustrated in Section~\ref{sec:GRDex} can be applied quite generally to sites. Whereas the geometric theory in Section~\ref{sec:Sites} describes the flat continuous functors rather explicitly, we can give a more internal description.
Of course, this is well known, and we shall only sketch it. The point we wish to make is that, using the ideas of Section~\ref{sec:GeoCons}, this description too can be understood as a geometric theory.

Suppose $(\catC, J)$ is a site.
The first step is to represent $\catC$ and $J$ as constants.
Thus $\catC$ will have sorts $\obj\catC$ and $\mor\catC$ for objects of objects and morphisms, together with functions for domain, codomain, etc.
There is an essentially algebraic theory of categories, and it is well known that the theories can be described geometrically using predicates for the graphs of the partial operations (such as composition of morphisms). However, now we know that pullbacks can be characterized geometrically, we can use them together with total operations.

The coverage $J$ can be represented in a similar way to the relations in a GRD-system.
Each cover $C=\{f_i\colon A_i\to A\mid i\in I_C\}\in J$ has a codomain $A$ in $\obj\catC$, and a set $I_C$ of morphisms $f_i$ with that common codomain.
Thus we can represent the structure with a function $J\to \obj\catC$, and a pair of functions $I \to J$ and $I \to \mor\catC$, suitably axiomatized, with $I=\sum_{C\oftype J}I_C$.

Note also that the site condition can be made geometric, albeit not directly, as it has some alternating quantifiers: $\Forall {C,g} \Exists{C'} \Forall {i'}\Exists {i}\ldots$.
The choice of $C'$ needs to be introduced as geometric structure, as a function of $C$ and $g$, and then the condition becomes
$\Forall {C,g}\Forall{i'\oftype C'(C,g)}\Exists {i}\ldots$,
which can be expressed as a geometric sequent.
The extra structure makes no difference to the theory $\pt{C,J}$.

Then the family $(X_A)$ of sets, indexed by objects $A$ of $\catC$, is represented as a sort $X$ with function symbol $p\colon X \to \obj\catC$.
The action by $\catC$ is then given by a function
$\alpha \colon X \times_{\obj\catC} \mor\catC \to X$, axiomatized to require the identity and composition laws for the action.

Now flatness can be axiomatized geometrically, with the external disjunctions of Section~\ref{sec:Sites} replaced by internal existential quantification,
and so too can continuity.

We have now done what is needed for relativization of sites: we have shown how a single geometric theory in the style of Section~\ref{sec:Examples} can be converted into a form that can be taken as an extension of a geometric theory, in the style of Section~\ref{sec:GeoCons}.

However, this is not quite the end of the story. In many ways the explicit description of a functor, with a family of sets indexed by objects of $\catC$, is more natural than the description using pullbacks.
In Section~\ref{sec:CtsFlatTT}
we shall say something about how to recover the naturality of the first style without losing the advantages of the second. 
It will be by a type-theoretic style in which object, as element of $\obj\catC$, is understood as a generalized element.

\subsection{Theory extensions as relativizations}
\label{sec:ExtnRel}

In each of the two previous subsections, we have constructed a theory extension $\thT_1$ of a smaller theory $\thT_0$.

In Section~\ref{sec:GRDex}, $\thT_0$ was for the space $\GRD$ of GRD-systems, and then for $\thT_1$ we extend this with $F$ axiomatized to be a point.
We may write this in the type-theoretic notation of contexts as $\Sigma\oftype \GRD, F\oftype\pt{\Sigma}$.
Thus we have two spaces with a model reduction map (forget $F$),
\[
  p\colon [\Sigma\oftype \GRD, F\oftype\pt{\Sigma}] \to \GRD
  \text{.}
\]

If we have an internal GRD-system $\Sigma$ in a Giraud frame $\baseS X$, then we have an internal geometric theory for $\pt{\Sigma}$ and hence a classifying topos $\baseS X \pt{\Sigma}$ (using $\baseS X$ instead of $\baseS$). Let us for the moment (in Section~\ref{sec:TTNotn} we shall have better proposals) write $X\pt{\Sigma}$ for the corresponding space.
A point of it is a pair $(x,F)$ where $x\oftype X$ (so that we get a space over $X$), and $F\oftype x^{\ast}\Sigma$.
But then if $\Sigma$ is represented as a map $\Sigma\colon X\to \GRD$ then these are equivalent to structures $(x,\Sigma', F)$ with $x\oftype X$, $\Sigma'\oftype \GRD$, $F\oftype \pt{\Sigma'}$ and $\Sigma'\cong\Sigma(x)$.
These are the points of a (bi-)pullback:
\[
  \begin{tikzcd}
  	X\pt{\Sigma}
  	  \ar[r]
  	  \ar[d]
  	& {[\Sigma\oftype \GRD, F\oftype\pt{\Sigma}]}
  	  \ar[d, "p"]
  	\\
    X
      \ar[r, "\Sigma"']
    & {\GRD}
  \end{tikzcd}
\]
Thus $p$ is generic locale-presented-by-GRD-system.
That is almost a generic locale (or localic map), except that the GRD-system is not uniquely determined by the locale.

Similarly for sites (Section~\ref{sec:RelSites}), we have a theory $\thT_0$ of sites, and and an extended theory
$\thT_1 = (\catC,J)\oftype\site, X\oftype\pt{\catC,J}$ of a site equipped with a flat continuous functor.

Thus we have shown, for two standard modes of presenting geometric theories, that there is a corresponding theory extension such that every classifier is a bipullback of the model reduction map.

What we want to argue now is that \emph{every} theory extension can be understood as a mode of presentation of geometric theories.

Suppose $\thT_1$ is an extension of $\thT_0$, and $a$ is a model of $\thT_0$ in $\baseS W$. We want to define a theory $\pt {a}$ relative to $\baseS W$.
First, $\thT_1$ is a theory relative to $\baseS$ and hence also to $\baseS W$. Next, relative to $\baseS W$, we can also syntacticize $a$ in a theory $\mathbb{A}$, equivalent to the one-point theory.
We take the disjoint union of these theories, giving the product of spaces, and adjoin isomorphisms between the $\thT_0$-ingredients in $\thT_1$ and the corresponding constants in $\mathbb{A}$.
This is our $\pt {a}$. Its models are the models of $\thT_1$ for which the $\thT_0$ reduct is isomorphic to $a$.

Just as for GRD-systems, we have a bipullback diagram
\begin{equation} \label{eq:extnBipb}
  \begin{tikzcd}
  	{W[\pt{a}] = a^\ast[\thT_1]}
      \ar[r]
      \ar[d]
    & {[\thT_1]}
      \ar[d, "p"]
    \\
    W
      \ar[r, "a"']
    & {[\thT_0]}
  \end{tikzcd}
\end{equation} 

\section{Bundles}
\label{sec:Bundles}
There are various different kinds of \glemph{bundle} defined in traditional topology, emphasizing different aspects.
However, what they have in common is the idea that a family of \emph{fibres,} indexed by points of a \emph{base space,} are ``bundled together'' to form a \emph{bundle space} (or \emph{total space}), with a map down to the base space, such that the inverse image of each point is isomorphic to the corresponding fibre.

Conceptually then, as in tangent bundles for example, a bundle $p\colon Y\to X$ is just a map, but thought of in an inverse way: for each point $x\oftype X$, we define a space, the fibre $p^{-1}(x)$ over $x$.

Point-set, there is still a mess to wade through if we are to get from the bundle view to the map $p$. We must form the point-set coproduct $\sum_{x\oftype X}p^{-1}(x)$, define an appropriate topology on it (\emph{not} in general the coproduct topology), and prove that that topology is correct on each fibre, and makes the obvious $p$ continuous.

The problem here is that we should like the fibre, somehow, to vary continuously with the base-point (think: tangent spaces).
Point-set, we have no notion of continuity for an assignment of spaces to base-points, as we do not have a point-set topological space of all spaces.
This is partly because the ``set'' of points (all spaces) would be large, but also because we need the concept of generalized spaces.

Point-free, with ``continuous'' just meaning ``geometric'', we do have a space of spaces as geometric theories, and we obtain a straightforward notion of bundle that leads directly from fibre assignment to bundle space.  

\begin{definition}
\label{def:bundle}
  A \emph{bundle} over a space $X$ is defined by a geometric construction of spaces $Y(x)$ (the fibres, presented as geometric theories) out of points $x\oftype X$.
\end{definition}

In fact, we have already seen a general discussion of this process in Section~\ref{sec:ExtnRel}. We can suppose that our description of the space is given by an extension of theories, $\thT_1$ extending $\thT_0$, with reduction map
$p\colon [\thT_1] \to [\thT_0]$.%
\footnote{
  It would suffice take those to correspond to site presentations, but it is often useful to use more particular instances of $p$, with particular properties -- such as GRD-systems for the localic case.
}
Then any \emph{point} $a$ of $[\thT_0]$ gives rise to a \emph{space} $\pt{a}$, that of $\thT_1$ models equipped with an isomorphism between the $\thT_0$-reduct and $a$.

Now look back at diagram~\eqref{eq:extnBipb}.
If $a$ is a global point, $W=1$, then the bipullback is equivalent to the fibre $p^{-1}(a)$ over point $a$.
Let us say, then, that for our generalized points, the (generalized) fibre is still given by the bipullback.
Then $p$ itself is the fibre over the generic point $\Id\colon [\thT_0]\to [\thT_0]$.

Rather than insisting on using ``$\pt{\cdots}$'' in every situation, we can assimilate the notation with Definition~\ref{def:bundle}.
This is because that definition gives us a theory extension, with $\thT_0$ a theory for $X$ and $\thT_1$ a theory for pairs $(x,y)$ with $x\oftype X$ and $y\oftype Y(x)$. 

\begin{example}
  Let us calculate the tangent bundle of the sphere $S^2$, treated as a subspace of $\Reals^3$: $S^2 = [x\oftype\Reals^3\mid x\cdot x = 1]$.
  Note that so far everything is geometric: $\Reals$, its arithmetic operations, the inner product on $\Reals^3$; and $S^2$ thus defined is clearly a subspace, as it is an equalizer.
  
  Now, for each $x\oftype S^2$, we can define its actual tangent plane $T_x S^2$ in $\Reals^3$ as
  $[y\oftype\Reals^3 \mid (y-x).x = 0]$.
  This is an affine subspace of $\Reals^3$, and comes equipped with a canonical zero point, namely $x$. Hence it is a (point-free topological) vector space.
  
  None of this should be at all surprising.
  The unexpected bit is that \emph{immediately} we have a bundle $p\colon T S^2=\sum_{x\oftype S^2} T_x S^2 \to S^2$:
  its existence follows from the geometricity of the construction of $T_x S^2$.
  In this case we have
  $T S^2 = [x,y\oftype \Reals^3 \mid x\cdot x = 1, (y-x)\cdot x = 0]$, and we know that this geometric description already describes the topology as well as the points.%
  \footnote{
    What we still don't have is a general point-free account of differential geometry to say why $S^2$ is a differential (let alone smooth) manifold and $T S^2$ its tangent bundle.
  }
  
  We shall say more about the notation $\sum_{x\oftype S^2}$ in Section~\ref{sec:TTNotn}, but it does \emph{not} denote a coproduct.
  The topology is plainly wrong for that. Neighbouring base points have neighbouring tangent planes, and we can see this as reflecting the fact that the indexing space $S^2$ is not discrete. 
\end{example}

\begin{remark}
  For AUs, the argument of Section~\ref{sec:ExtnRel} does not quite work, because we do not have the ability to syntacticize constants.
  Nonetheless, \cite{Vickers:AUSk} shows how an extension of arithmetic theories can be treated as a bundle for toposes, and in a base-free way. Over any elementary topos (with nno), any model of $\thT_0$ gives rise to a space.
\end{remark}

\subsection{Type-theoretic notation}
\label{sec:TTNotn}
We now seek ways to use the notation of dependent type theory to denote spaces.

In the bundle notation of Definition~\ref{def:bundle}, we have rewritten the extension $\thT_1$ to show how it comprises the $\thT_0$-reduct ($x$, say), together with some extra $Y(x)$ dependent on it.
We can express this in a notation of dependent type theory, as
\[
  x\oftype[\thT_0] \vdash Y(x) \text{ space}
\]

Type theory then suggests a dependent sum notation for $[\thT_1]$ as $\sum_{x\oftype[\thT_0]} Y(x)$, as its points are pairs $(x,y)$ where $x\oftype [\thT_0]$ and $y\oftype Y(x)$.
In this convenient notation we have the means to denote relativized theories.

The $x$ there is \glemph{generic}, but we can substitute any \glemph{specific} point. Suppose we are using $\thT_0$ and $\thT_1$ to define a bundle over $W$ as in diagram~\eqref{eq:extnBipb}, with a specific point $a\colon W \to [\thT_0]$, so we are saying
\[
  w\oftype W \vdash Y(a(w)) \text{ space}
\]

We can then write the fibre in a bipullback square
\begin{equation} \label{eq:DTTpb}
  \begin{tikzcd}
    {\sum_{w\oftype W}Y(a(w)) = a^\ast (\sum_{x\oftype [\thT_0]} Y(x))}
      \ar[r]
      \ar[d]
    & {\sum_{x\oftype [\thT_0]} Y(x)}
      \ar[d]
    \\
    W
      \ar[r, "a"']
    & {[\thT_0]}
  \end{tikzcd}
\end{equation}

Note that these sums, such as $\sum_{x\oftype [\thT_0]} Y(x)$, are \emph{not} coproducts in general.
The geometric construction of $Y(x)$ out of $x$,
which gives an internal space in $\baseS [\thT_0]$,
provides enough information for a topology across the fibres,
giving cohesion between neighbouring fibres.

Diagram~\eqref{eq:DTTpb} gives the essential reason why point-free topology can be thought of as a dependent type theory, because the top-left corner shows \emph{substitution} ($\sum_{w\oftype W}Y(a(w))$) as \emph{pullback} ($a^\ast (\sum_{x\oftype [\thT_0]} Y(x))$).
Nonetheless, there are various features that make it an unusual type theory.

The first is that function types and $\prod$-types can exist \emph{only in particular cases,} and must be argued case by case. They cannot be a general part of the logic, because the categories of spaces (generalized or ungeneralized) are not cartesian closed.
For locales, $X$ is exponentiable iff it is locally compact, and not all are.

Even where function types do exist, they may fail to behave as expected.
For example, if $X$ and $Y$ are discrete then the exponential $Y^X$ exists, but will not itself be discrete in general -- it has the point-open topology, equivalently the product topology for the $X$-fold product of copies of $Y$. Thus even the category of discrete spaces (sets) is not cartesian closed geometrically~\cite[3.1.1]{TopCat}.

In Section~\ref{sec:Exp} we explore some of the reasoning for these function types that do exist.
The resulting mathematics has some notable surprises.
For instance (Example~\ref{ex:0Power}),
suppose $\phi$ is a subset of 1 (\ie\ an open).
Then, allowing for the fact that $0^\phi$ is not discrete, we find that $0^{0^\phi}\cong\phi$ -- the double negation law of logic holds, geometrically!

Some other uses of exponentials, the \emph{bagspaces,} are mentioned in Section~\ref{sec:OpfibFib}.

A second departure from usual type theory is with identity types. This can reasonably be defined (for a space $X$) as
$[x,x'\oftype X, \phi\colon x \cong x']$.
However, the category of spaces is in fact a 2-category, so we also need to refer to an asymmetric identity type
$[x,x'\oftype X, \phi\colon x \to x']$.
This is incompatible with the usual type-theoretic notions of path transport, and for good reasons. We need to know whether path transport, if it even exists, goes covariantly or contravariantly, and Section~\ref{sec:OpfibFib} shows that all these phenomena can exist with spaces.

Finally, a third departure is with universes.
As we have seen, for any theory extension, of $\thT_0$ by $\thT_1$, the base $[\thT_0]$ can be understood as a universe in that its points can be understood as spaces.
Thus we get many different universes for different kinds of spaces.
On the other hand, we don't get a hierarchy of universes forced on us by Russell's paradox or the like, as Russell's paradox cannot be expressed geometrically.
Thus any notion of universe is very different from that of, say, Homotopy Type Theory.

\subsection{Opfibrations and fibrations}
\label{sec:OpfibFib}

We briefly discuss two examples of bundles that have interesting interactions with the 2-cells for spaces.
These are just two simple examples of much broader phenomena~\cite[B4.4]{Elephant1}.

The first is that of discrete spaces, using the \emph{object classifier} $\Set$, for the theory with one sort and nothing else. Its models are \glemph{sets}.
Given a set $X$, the points of the corresponding discrete space are just the elements of $X$, so we may write $X$ also for the space. 
Our bundle space is then $[X\oftype\Set, x\oftype X]$, or $\sum_{X\oftype\Set} X$.
Using~\cite{JoyalTier}, the bundle map is a local homeomorphism, and in fact local homeomorphisms are exactly its bipullbacks.

What we want to remark on here is how the bundle interacts with 2-cells.
Suppose we have two sets $X$ and $X'$, and a function $f\colon X\to X'$. By the 2-categorical aspects of the definition of classifying topos, this is a 2-cell between two maps $X$ and $X'$ into $\Set$.
Clearly this gives a map, covariantly, between the corresponding discrete spaces, in other words the fibres over $X$ and $X'$.
This works out as saying the bundle is an \emph{opfibration} in the 2-category of spaces.

Interestingly, the opfibration property does not depend on a base $\baseS$. \cite[B4.4]{Elephant1} proves it for the whole 2-category of elementary toposes.
This is a much stronger result than simply saying it for $\baseS$-toposes, as the 2-cells there have to be the identity over $\baseS$ and so there are far fewer of them.
Thus for an object $X$ of $\baseS$ there is still non-trivial content in saying that the map $X\to 1$ is an opfibration, even though $1$ has no non-trivial endomorphisms.
\cite{Vic:FibCBFT} shows how to prove this in a relatively simple way, exploiting the base-free nature of AU arguments.

Our second example is that of spectral spaces, of the form $\Spec(L)$ for a distributive lattice $L$.
Its points are the prime filters of $L$, and these can be axiomatized geometrically.
Evidently we get a theory extension (distributive lattice extended with prime filter) and hence a bundle.
The bundle space can be written
$[L\oftype\DL, y\oftype \Spec(L)]$ or $\sum_{L\oftype\DL} \Spec L$.

This time, the bundle is a fibration.
If we have a distributive lattice homomorphism $f\colon L\to L'$, then we have a map \emph{contravariantly} between the spectra, $\Spec(f)\colon \Spec(L')\to \Spec(L)$.

Note that the general GRD bundle for localic maps is neither fibration nor opfibration.

An important use of these considerations is in the theory of \emph{bagspaces,} spaces of ``bags'', indexed families, of points of some $Z$.
(See~\cite[B4.4]{Elephant1}, where, following~\cite{GeoThDB}, they are called ``bagdomains''.)
The possible indexing spaces will be described by some extension
$x\oftype X \vdash I(x) \text{ space}$,
and then the bagspace is $\sum_{x\oftype X} \prod_{i\oftype I(x)} Z$ (or $\sum_{x\oftype X} Z^{I(x)}$).
Categorically, this is a \emph{partial product} of $Z$ against the extension map $\sum_{x\oftype X} I(x) \to X$.
Whether these exist depends on the nature of the extension, and for a proper 2-categorical account the discussion depends on it being a fibration or an opfibration.
They do exist for indexation both by sets and by spectral spaces, as well as for many much more general examples.

See also \cite{SVW:GelfandSGTGM}, for a discussion of how the distinction between fibrational and opfibrational led to an otherwise puzzling difference between two topos approaches to quantum physics.

\subsection{Continuous flat functors, type theoretically}
\label{sec:CtsFlatTT}

In Section~\ref{sec:RelSites} we discussed the internal notion of functor, and how that could be represented directly using an ability to have pullbacks as a sort constructor in geometric theories.
We made out that this was advance on the previous notation of Section~\ref{sec:Sites}, which was more directly in terms of a functor.
However, in practice the external functor notation is a more natural way to think about the structures, and what we do now is to show that type theoretic notation can be used to provide this even for internal functors.

What we do relies on the fact that any discrete space $X$ is exponentiable, so the exponential $\Set^X = \prod_{x\oftype X} \Set$ exists. (More generically, this comes down to the bagspaces of Section~\ref{sec:OpfibFib}, with $X$ as indexing set.)
In fact $\Set^X$ is equivalent to $[Y\oftype \Set, p\colon Y\to X]$.
This gives exactly the bundle view of the map $p$.

We shall use the notation of dependent types with this, though note that this is a dependent type theory for the internal mathematics of Giraud frames, and not for the mathematics of spaces.
If we have $(Y(x))_{x\oftype X}\oftype \prod_{x\oftype X}\Set$,
then the bundle set $Y$ is $\sum_{x\oftype X}Y(x)$,
a coproduct in the internal sense using indexed categories.
The symbol $x$ itself, bound here by $\sum$, is generic. It may be instantiated by any generalized element of $X$,
giving a pullback.

We can now describe the theory of categories in a manner familiar to type theorists, using the following symbols and axioms.
\begin{itemize}
\item
  $\obj\catC\oftype\Set$. We shall write $i, j$, etc. for typical elements.
\item
  $\catC=(\catC\homsp{i}{j})_{i,j\oftype \obj\catC}\oftype\prod_{i,j\oftype\obj\catC} \Set$, a type that we know to be geometric.
  Note how we have given ourselves some flexibility in notation.%
  \footnote{
    The particular notation $\catC\homsp{i}{j}$ for homsets is experimental and derived from tensor notation in physics.
    Here it is based on the fact that $\catC$ is a $\catC$-$\catC$-bimodule (profunctor).
    Superscripts show left, presheaf action, contravariant, while subscripts show right, diagram action, covariant.
    (Composition is diagrammatic order.)
  }
  The type of morphisms is then $\sum_{i,j\oftype\obj\catC} \catC\homsp{i}{j}$.
\item
  A family of constants $\Id_i\oftype\catC\homsp{i}{i}$.
  What this means is that we have two points of the space $\prod_{i\oftype\obj\catC} \Set$, namely $(1)_{i\oftype\obj\catC}$ and $(\catC\homsp{i}{i})_{i\oftype\obj\catC}$,
  and $i\mapsto \Id_i$ is a specialization morphism between them.
  We shall commonly omit the subscript in $\Id_i$.
\item
  A family of functions $\cdot_{ijk}\colon \catC\homsp{i}{j} \times \catC\homsp{j}{k} \to \catC\homsp{i}{k}$, written $(u,v)\mapsto uv$.
  (Note our preference for diagrammatic order of composition.)
\item
  Axioms $\Id u = u = u\Id$ and $(uv)w = u(vw)$.
  These really represent equations between specialization morphisms between points of $\prod$-types in an obvious way.
\end{itemize}

An internal diagram $X$ over $\catC$ can be described in a similar style.
It is a bundle $(X_i)_{i\oftype\obj\catC}\oftype \prod_{i\oftype\obj\catC}\Set$,
equipped with actions $X_i \times \catC\homsp{i}{j} \to X_j$,
written $(x,u)\mapsto xu$,
satisfying the axioms $x\Id = x$ and $x(uv)=(xu)v$.

It is flat if the following maps are all epi.
\begin{gather*}
	\begin{split}
	  & \sum_{k}X_k \to 1
	\\
	  & \sum_{k}
		X_k\times\catC\homsp{k}{i}\times \catC\homsp{k}{j} 
		\to X_i \times X_j
		\quad \left((x,u,v) \mapsto (xu, xv)\right)
	\\
	  & \sum_{k}
		\left[(x',u,v,w)\oftype X_k \times \catC\homsp{k}{i} \times\catC\homsp{i}{j} \times
		\catC\homsp{i}{j} 
		\mid uv=uw \right]
	\\
	  & \quad\quad\quad \to \left[(x,v,w)\oftype X_i \times \catC\homsp{i}{j} \times \catC\homsp{i}{j}\mid xv=xw\right]
	  \quad
		\left((x',u,v,w) \mapsto (x'u,v,w)\right)
	\end{split}
\end{gather*}

We leave to the reader the exercise of extending this to coverages and continuous flat functors.

\section{Some exponentials that exist} \label{sec:Exp}

The type theory of spaces cannot include arrow types (or dependent product types more generally), as few categories of point-free spaces are cartesian closed.
In $\Loc$, a locale $X$ is exponentiable iff it is locally compact -- that is to say, by Hyland's Theorem, its frame $\Omega X$ is a continuous lattice.

For example $\Nat$ is locally compact, as is any set (discrete space), so the exponential $\Nat^{\Nat}$ (Baire space) exists, but is not itself locally compact.
In particular we see that $\Nat^{\Nat}$ exists \emph{but is not a set:}
the Baire space topology is not discrete.
Moreover, higher-order types such as $\Nat^{\Nat^{\Nat}}$ do not exist.

If $X$ and $Y$ are arbitrary sets, then the exponential $Y^X$ has for its points the functional (total and single-valued) relations from $X$ to $Y$, and these can be axiomatized geometrically.

The crucial point to recognize is that this exponential $Y^X$, with its normally non-discrete topology, is \emph{not} the exponential as calculated in a Giraud frame where $X$ and $Y$ live -- to get that one has to take the non-geometric \glemph{discrete coreflection}.
Hence the mathematics of the spatial exponentials is different from the familiar mathematics of elementary toposes.
This shows up particularly with higher-order types,
in special cases where $Y^X$ is also exponentiable.
Then its non-discrete topology shows that, even classically, the spatial $Z^{Y^X}$ is qualitatively different from the ``set-theoretic'' one in that it has fewer points, because it admits only \emph{continuous} maps from $Y^X$ to $Z$.

In this section we explore some surprising features of this reinterpreted exponential.
A simple example (Example~\ref{ex:2Power}) is that $2^{2^X}$ is the free Boolean algebra on $X$. This is smaller than expected classically, and follows once one understands that $2^X$ has the product topology of the $X$-fold product of copies of 2 -- it is a \glemph{Stone space}, analogous to Cantor space $2^\Nat$.
Some examples are quite shocking to a constructive mathematician. For instance (Example~\ref{ex:0Power}), logically the double negation law $\neg\neg\phi = \phi$ holds, once one reinterprets the Heyting negation $\neg\phi$ as an exponential $0^\phi$ (as is usual in type theory).

Since our main aim is to illustrate the surprising behaviours, we shall be content with frame-based proofs.
Nonetheless, it is to be hoped that fully geometric proofs can be found.
\subsection{Two results on discrete and Stone spaces}
\label{sec:TwoResultsDS}

\begin{proposition} \label{prop:SPowerD}
  Let $X$ be discrete and $Y=\Spec B$ \glemph{Stone}. Then $Y^X$ is Stone.
\end{proposition}
\begin{proof}
  There seems to be a variant of Tychonoff's Theorem at work here: any product of Stone spaces is Stone.
  
  We shall calculate more explicitly.
  A map from $X$ to $\Spec B$ is an $X$-indexed family of prime filters of $B$, hance a relation $\theta$ from $X$ to $B$ such that, for each $x\in X$, the set $(x,-)=\{b\in B\mid x\mathrel{\theta} b\}$ is a prime filter.
  But any such relation is decidable, with $(x,b)\notin\theta$ iff $x\mathrel{\theta}\neg b$,
  and so corresponds to a map from $Y\times B$ to 2; and the conditions for $(x,-)$ to be a prime filter can be expressed coherently (no infinite disjunctions).
  It follows that maps from $X$ to $\Spec B$ are geometrically in bijection with prime filters of the Boolean algebra generated by $X\times B$, subject to axioms expressing the prime filter conditions.
\end{proof}

\begin{proposition} \label{prop:DPowerS}
  Let $X = \Spec A$ be Stone and $Y$ discrete. Then $Y^X$ is discrete.
\end{proposition}
\begin{proof}
  $\Omega Y = \power Y$ is the frame generated by $Y$ subject to axioms as in Proposition~\ref{prop:setToSort},
  and $\Omega X = \Idl A$.
  It follows that maps from $X$ to $Y$ are equivalent to subsets $\theta\subseteq Y\times A$ such that
  \begin{enumerate}
  \item
    if $y\theta a \geq a'$ then $y\theta a'$,
  \item
    $y\theta 0$,
  \item
    if $y\theta a$ and $y\theta a'$ then $y\theta (a\vee a')$,
  \item
    if $y\theta a$ and $y'\theta a'$ then $a\wedge a'=0$ or $y=y'$, and
  \item
    there is some $\theta_0\in\fin\theta$
    with $\unjoin_{y\theta_0 a} a = 1$.
  \end{enumerate}
  Let us (temporarily) call such a $\theta$ a \emph{map relation}.
  
  Now suppose $\phi_0 \in \fin(Y\times A)$ has
  $\unjoin_{y\phi_0 a} a = 1$, and also satisfies condition (4) with $\phi_0$ substituted for $\theta$.
  Define $y \overline{\phi}_0 a$ if there is some $I\in\fin\phi_0$ such that $y=z$ for all $(z,b)\in I$ and $a\leq \unjoin_{(z,b)\in I} b$.
  Then $\overline{\phi}_0$ is a map relation.
  To prove condition (4), suppose we have $y\overline{\phi}_0 a$ and $y'\overline{\phi}_0 a'$, with corresponding $I$ and $I'$.
  Then we have
  $\Forall {(z,b)\in I} \Forall {(z',b')\in I'} (z=z' \vee (b\wedge b')=0)$.
  By properties of Kuratowski finite sets, it follows that
  $\Forall {(z,b)\in I} \left( (\Exists {(z',b')\in I'} z=z') \vee (\Forall{(z',b')\in I'} b\wedge b'=0)\right)$,
  and then that
  $\Exists {(z,b)\in I} \Exists {(z',b')\in I'} z=z' \vee \Forall{(z,b)\in I}\Forall{(z',b')\in I'} b\wedge b'=0$.
  In the first case we have $y=z=z'=y'$, and in the second $a\wedge a'=0$.
  
  We claim that, for any map relation $\theta$, if $\phi_0\subseteq\theta$, then $\theta = \overline{\phi}_0$.
  The $\subseteq$ direction is clear.
  For the reverse, suppose $y\theta a$.
  We have $a = \unjoin_{y' \phi_0 a'} (a\wedge a')$.
  For each $y' \phi_0 a'$ we have $y' \theta a'$, and hence either $y=y'$ or $a\wedge a'=0$ so we can find a partition $\phi_0 = I \cup J$ such that $y=y'$ for every $(y',a')\in I$, and $a\wedge a' =0$ for every $(y',a')\in J$.
  It follows that
  \[
    a = \unjoin_{y' \phi_0 a'} (a\wedge a') = \unjoin_{(y,a')\in I} (a\wedge a') \leq \unjoin_{(y,a')\in I} a'
    \text{,}
  \]
  and so $y\overline{\phi}_0 a$.
  
  We deduce that, if $\theta$ is a map relation, equipped with $\theta_0$ as in condition (5), then $\theta = \overline{\theta}_0$ and so is determined by $\theta_0$.
  
  The required properties of $\theta_0$ (stated in (5) and inherited from (4)) are geometric formulae, and hence define a subset of $\fin(Y\times A)$.
  However, different $\theta_0$s may define the same $\overline{\theta}_0$, so we must quotient out.
  By our above claim, the appropriate equivalence relation is $\theta_0 \equiv \phi_0$ if $\theta_0 \subseteq \overline{\phi}_0$, and this is a geometric formula.
  
  Thus the points of $Y^X$ are geometrically equivalent to the elements of a set, and so $Y^X$ is discrete. 
\end{proof}

\subsection{Discrete Stone spaces} \label{sec:DiscStone}

From Propositions~\ref{prop:SPowerD} and~\ref{prop:DPowerS},
we see that if $Y$ is both discrete and Stone, then the transformation $X\mapsto Y^X$ alternates between discrete and Stone -- and thus preserves local compactness in those two cases.

\begin{proposition} \label{prop:DiscStone}
  A set $Y$ is a Stone space iff it is (Kuratowski) finite and has decidable equality.
\end{proposition}
\begin{proof}
  $\Rightarrow$:
  Finiteness follows from compactness. $Y$ is covered by the singletons $\{y\}$, and hence by a finite subset of them.
  Equality is decidable because $Y$ is regular, so that equality is closed.
  
  $\Leftarrow$:
  Finiteness means that $\fin Y$, the free semilattice (under $\cup$) over $Y$, has a top element, and decidable equality ensures that it also has binary intersections.
  In this case it is also a Boolean algebra.
  Then the elements of $Y$ are geometrically in bijection with the prime filters of $\fin Y$, so $Y = \Spec(\fin Y)$ is Stone.
\end{proof}

\begin{remark}
  Suppose $Y$ is discrete Stone.
  By finiteness it is $\{y_1,\ldots y_n\}$ for some $n\in\Nat$, and by decidable equality the $y_i$s can be chosen distinct.
  Then $n$ is unique, and provides a well defined finite cardinality for $Y$.
  
  However, $Y$ is not necessarily isomorphic to $\{1,\ldots n\}$.
  This is because our ability to choose an order for the distinct elements $y_i$ is only local.
  An easy counterexample is the twisted double cover $S^1\to S^1$ of the circle. It has cardinality 2 (every fibre has two elements), but it is not globally isomorphic to 2.
  For further discussion of the geometric issues here, see~\cite{SFP}.
\end{remark}

\begin{proposition} \label{prop:DSPowerD}
  Let $Y$ be discrete Stone, and $X$ discrete.
  Then $Y^X$ is Stone, and its Boolean algebra $B$ of clopens is generated by $X\times Y$, subject to the following relations for each $x\in X$:
  \begin{gather*}
      1 \leq \unjoin_{y\in Y} (x,y) \\
      (x,y)\wedge (x,y') \leq 0 \quad \text{(if $y\neq y'$)}
  \end{gather*}
\end{proposition}
\begin{proof}
  A prime filter of $B$, in other words a BA-homomorphism from $B$ to 2, is a \emph{decidable} subset of $X\times Y$ that, in respecting the conditions, is a functional relation from $X$ to $Y$.
  But every functional relation $\theta$ here is decidable, since $x(\neg\theta)y$ is the geometric formula
  $\Exists {y'} (y\neq y' \wedge x\theta y')$.
  Hence functions $X\to Y$ are equivalent to prime filters of $B$. 
\end{proof}

\begin{proposition} \label{prop:DSPowerS}
	Let $Y$ be discrete Stone, and $X=\Spec A$ Stone.
	Then $Y^X$ is discrete.
	It is the set of $\theta_0 \in \fin(Y\times A)$ satisfying the following conditions, all geometric formulae.
	\begin{enumerate}
    \item
      $\Forall {y\in Y} \Exists {a\in A} y\theta_0 a$
    \item
      $\Forall {(y,a), (y',a') \in \theta_0} (y\neq y' \vee a=a')$
    \item
      $\unjoin_{y\theta_0 a} a = 1$
    \item
      $\Forall {(y,a), (y',a') \in \theta_0} (y = y' \vee a\wedge a'=0)$
	\end{enumerate} 
\end{proposition}
\begin{proof}
  With $Y$ Stone, the finite $\theta_0$s appearing in the proof of Proposition~\ref{prop:DPowerS} can be worked into a canonical form satisfying the above conditions.
\end{proof}

\begin{example} \label{ex:2Power}
  Let $Y=2$, clearly discrete Stone.
  
  If $X$ is discrete, then $2^X$ is Stone. Its points are the decidable subsets of $X$, and $\Clop(2^X)$ is the free Boolean algebra over $X$.
  
  Note that $2$ is itself a Boolean algebra, and it follows that $2^X$ is a Stone Boolean algebra. (Compare with the duality between sets and Stone Boolean algebras, classically proved in~\cite[VI.3.2]{StoneSp}.)
  
  Now suppose instead that $X=\Spec A$ is Stone.
  Then $2^X$ is discrete. Its points, the clopens of $X$, are the elements of $A$: $2^X=\Clop X$.
  
  Because local compactness is preserved, we can repeat the exponentiation, to find that if $X$ is discrete then $2^{2^X}$ is the Boolean algebra freely generated by $X$.
  
  It follows that $2^{2^{-}}$ is the functor part of a monad.
  Of course, this is already well known, as a \emph{continuation monad,} in cartesian closed settings,
  and has been studied in elementary toposes using the exponentiation there.
  But note that it gives a \emph{different answer.} This is clear even for classical sets, where the usual $2^{2^X}$ has a much greater cardinality in the case of infinite $X$.
  It would be interesting to understand whether the geometric $2^{2^X}$ can be applied in settings where the continuation monad is used.  
\end{example}

\begin{example} \label{ex:0Power}
  Let $Y = 0$, again clearly discrete Stone.
  
  If $X$ is discrete, then $0^X$ is Stone. Its Boolean algebra is freely generated by $\emptyset$, subject to relation $1\leq 0$ for each $x\in X$.
  This is a quotient of $2=\{0,1\}$, and we can calculate the corresponding congruence explicitly as
  $\{(0,0), (1,1)\}\cup \phi\times\{(0,1), (1,0)\}$
  where $\phi$ is the image of $X$ in 1.
  Thus if $0=1$ in the Boolean algebra then $\phi$ holds.
  
  $0^X$ is a subspace of 1, and in fact is the closed complement of the open subspace given by $\phi$. (See~\cite{ArithInd} for further discussion of how closed subspaces are Stone.)
  
  If on the other hand $X = \Spec A$ is Stone, then $0^X$ is discrete, $\{\emptyset\mid 0=1 \text{ in } A\}$.
  
  Putting these together, if $X$ is discrete then $0^{0^X}$ is the image of $X$ in 1,
  and $0^X$ is its closed complement.
    
  Shockingly, this means that if $X$ is already a subobject of 1 (a truthvalue), then $0^{0^X}=X$. In type theory, logical negation of $X$ is in effect defined as the exponential $0^X$, so what we see is that the double negation law of logic holds!
  Of course, type theory applies a hidden \glemph{discrete coreflection} to $0^X$, which explains the different answer.
  
  We can go further. For each locale, the subspaces form a lattice, in which each open has a Boolean complement -- the corresponding closed subspace.
  Thus we also have excluded middle $X \vee 0^X = 1$, once we understand that the join is taken of sub\emph{spaces}, not of sub\emph{sets} (subobjects) of 1.
\end{example}

\begin{example} \label{ex:SierpClosed}
  The key to understanding Example~\ref{ex:0Power} is that $0^X$ is not a sub\emph{set} of 1, which would be open, but a closed sub\emph{space}; and its topology is Stone, not discrete.
  
  This is hard to visualize from the perspective of point-set topology, but we can see it in more concrete form using Sierpinski space $\Sierp$, whose points are subsets of 1.
  Clearly it has two classical points $\bot=0$ and $\top=1$, and in fact every point is a directed join of these.
  Hence $\Sierp$ is the ideal completion of the 2-element poset $\twopos=\{\bot\sqsubseteq\top\}$.
  It follows by presheaf theory that $\baseS\Sierp$ is the arrow category $\baseS^{\rightarrow}$.
  The generic point $P$ is $(0\to 1)$, a subsheaf of $1=(1\to 1)$.
	
  If we calculate the discrete space for $P$, as a local homeomorphism over $\Sierp$, we find stalks $P(\bot)=0$ and $P(\top)=1$.
  Its closed complement $P'=0^P$ has $P'(\bot)=1$ and $P'(\top)=0$.
  This is very much not a local homeomorphism, as it does not have the opfibration property (Section~\ref{sec:OpfibFib})-- there cannot be a fibre map from $P'(\bot)$ to $P'(\top)$.
  As a Stone space $\Spec B$, $B$ is a sheaf of Boolean algebras with $B(\bot)=2$ and $B(\top)=1$.
  $B(\bot)$ has exactly one prime filter, $\{1\}$, but $B(\top)$ has none. A prime filter would have to both contain 1 and not contain 0, which is impossible since they are equal.
	
  Let us calculate the (non-geometric) set of points of $0^P$, its \glemph{discrete coreflection}.
  This is the Heyting negation $\neg P$, and can also be viewed as the interior of $0^P$.
  Thinking of bundles, it is an approximation of $0^P$ by a local homeomorphism, with map $\neg P \to 0^P$.
  Looking at fibres over $\top$, we find that the map $(\neg P)(\top)\to (0^P)(\top)$ forces $(\neg P)(\top)=0$; but then the opfibration property for local homeomorphisms forces $(\neg P)(\bot)=0$.
  It follows that $\neg P$ is empty, and has lost too much information to be a good approximation to $0^P$.
\end{example}

\section{Conclusions} \label{sec:Conc}
This article is not so much finished work as an invitation to further research. We have explained how, if one's mathematics is geometric, then it can be expressed in terms of point-free spaces (even generalized ones) and conducted in a pointwise manner.

Insofar as one's mathematics \emph{is} geometric, this is a big benefit, especially the prospect of dealing with bundles in a natural way as continuous space-valued maps.
An example of the style is~\cite{NgVic:PtfreeRE}, which defines real exponentiation and logarithms point-free.
However, the geometricity is a big proviso.

How much mathematics can be done geometrically, with everything continuous?

Not all of it, obviously. For instance, a large part of existing real analysis allows for the possibility of discontinuous functions on the reals, and set theory studies many sets that cannot be constructed geometrically.
Nonetheless, it's not impossible that a central core of continuous, geometric mathematics essentially covers the most important applications.
We might surmize that physics is intrinsically continuous -- few would believe that the Banach-Tarski construction has a reality in physics.
Also, Scott's work on denotational semantics postulates that computable functions are continuous with respect to a suitable non-discrete topology.
This was crucial in avoiding a cardinality blow-up in his semantics for the untyped $\lambda$-calculus.

The biggest open question for the geometric programme is algebraic topology.
A large part of Grothendieck's vision for toposes as generalized topological spaces is that they are structures \emph{for which one can calculate algebraic topology,}
for example sheaf cohomology.
If one takes an Abelian group $G$ internal in $\baseS X$ as group of coefficients, then one calculation is to take an injective coresolution of $G$ in $\baseS X$, apply the global sections functor $!_\ast$ to get a complex of groups in $\baseS$, and extract the cohomology groups from that.
This process is thoroughly non-geometric, and -- in my own limited understanding -- may depend on $\baseS$ being classical sets.
Nonetheless, one might hope for specific values of the cohomology groups to be robust enough to have a geometric characterization.

Many open questions centre on the transfer from geometric to arithmetic, the coherent (finitary) fragment of geometric logic, augmented with list objects (including nno's) to capture a significant part of the infinities in a base-free way.
Already in~\cite[6.1]{TopCat} it was remarked that, despite the overall geometric reasoning, there were still instances where the intuitionistic logic of elementary toposes was used to prove results stated geometrically: this would be a block to an arithmetic treatment.
Nonetheless, we may hope to find ways round this, such as the methods of~\cite{ArithInd}.

A big example of such a block lies in hyperspaces (powerlocales), which seem to be much more important in point-free topology than point-set.
(See, \eg, \cite{CViet}.)
\cite{PPExp} showed that the powerlocale constructions, normally applied to frames, can be reduced to geometric constructions on GRD-systems. This can be used to show that the constructions are geometric -- preserved by pullback along geometric morphisms.
This is related to the \emph{equivariant} property of the symmetric monad on toposes (\cite[6.4.4]{BungeFunk}; see also~\cite{ConnFT}).

\newcommand{\hyp}{\mathsf{P}}
This raises the question of whether a hyperspace construction $\hyp$ can be described purely in terms of GRD-systems, without reference to the frames they present. Certainly we can define maps $\hyp\colon\GRD\to\GRD$, and also define the corresponding unit and multiplication maps for $\hyp$ as a monad.
What is still open is how to capture the functoriality of $\hyp$ \emph{with respect to locale maps.}
They are automatically functorial with respect to GRD-homomorphisms, but that is different.
Treatments such as \cite{PPExp} are hybrid, making use of the frames in various places.

Finally, let us be clear that describing spaces as geometric/arithmetic theories can hardly be the final word in point-free topology.
For a start, it does not encompass any higher-dimensional (homotopy) features.
\cite{AnelJoyal:Topologie} does, using an explicit formal duality ``Topo-logie'', and in our current state of knowledge it is hard to see how else you would do it.

Another sticky point geometrically lies in the classifier $BG$ for a topological (= localic) group $G$.
(See \cite{Moer:ClassTCGI}.)
If $G$ is discrete, then $BG$ classifies the geometric theory of torsors over $G$, and they are equivalent to principal $G$-bundles.
More generally, the sheaves over $BG$ (the objects of the Grothendieck topos) are the continuous $G$-sets,
and the category of them does form a \glemph{Giraud frame}.
However, the proof of the set of generators (from which one might extract a geometric theory) does not appear to be geometric.
Hence we do not know of a convenient description of the points.

Moreover, the passage from $G$ to $BG$ can lose information.%
\footnote{
  -- in the case where $G$ is not \'etale complete.
}
For example, if $G$ is connected then continuous $G$-actions are all trivial, so $BG\simeq 1$.
(Suppose $X$ is a $G$-set. If $x\in X$ then the map $g\mapsto gx$ from the connected $G$ to the discrete $X$ must be constant, and so $gx = 1x = x$.) 
Topos theory does not provide a generalized space of principal $G$-bundles for such a $G$;
and in fact a simple argument of Lurie~\cite{Lurie:ClassITTLG} shows that this cannot be done even with the higher dimensional information of $\infty$-toposes.

Even if there are essential spaces (such as that of principal bundles over a connected group) that cannot be captured by the geometric/arithmetic theories,
one might still hope that there is a role for them in bootstrapping:
laying some minimal foundations for higher levels of construction.  

\bibliographystyle{amsalpha}

\providecommand{\bysame}{\leavevmode\hbox to3em{\hrulefill}\thinspace}
\providecommand{\MR}{\relax\ifhmode\unskip\space\fi MR }
\providecommand{\MRhref}[2]{%
	\href{http://www.ams.org/mathscinet-getitem?mr=#1}{#2}
}
\providecommand{\href}[2]{#2}

\end{document}